\documentclass{amsart} 

\usepackage{amsmath,amssymb,amsthm,amsfonts,amsxtra,mathtools}
\usepackage{enumerate,verbatim,mathrsfs,comment,color,url}
\usepackage[all,2cell,ps]{xy}
\usepackage[pagebackref]{hyperref}
\usepackage{graphicx}
\usepackage{tikz-cd}
\usepackage{cleveref}
\usepackage{verbatim}

\newcommand{\mycomment}[1]{}

\DeclareMathOperator{\cx}{cx}
\DeclareMathOperator{\depth}{depth}

\DeclareMathOperator{\Ext}{Ext}

\DeclareMathOperator{\gdim}{G-dim}

\DeclareMathOperator{\Hom}{Hom}

\DeclareMathOperator{\id}{id}

\DeclareMathOperator{\injcx}{inj\,cx}

\DeclareMathOperator{\pd}{pd}

\DeclareMathOperator{\Spec}{Spec}

\DeclareMathOperator{\Supp}{Supp}

\DeclareMathOperator{\Tor}{Tor}

\DeclareMathOperator{\type}{type}

\renewcommand{\ge}{\geqslant}
\renewcommand{\le}{\leqslant}

\newcommand{\fm}{\mathfrak{m}}
\newcommand{\fn}{\mathfrak{n}}
\newcommand{\fp}{\mathfrak{p}}

\renewcommand{\iff}{if and only if }

\theoremstyle{plain}
\newtheorem{theorem}{Theorem}[section]
\newtheorem{lemma}[theorem]{Lemma}
\newtheorem{proposition}[theorem]{Proposition}
\newtheorem{corollary}[theorem]{Corollary}

\newtheorem{innercustomtheorem}{Theorem}
\newenvironment{customtheorem}[1]
{\renewcommand\theinnercustomtheorem{#1}\innercustomtheorem}
{\endinnercustomtheorem}

\theoremstyle{definition}
\newtheorem{definition}[theorem]{Definition}
\newtheorem{conjecture}[theorem]{Conjecture}

\newtheorem{example}[theorem]{Example}

\newtheorem{notation}[theorem]{Notation}

\newtheorem{para}[theorem]{}
\newtheorem{question}[theorem]{Question}

\theoremstyle{remark}
\newtheorem{remark}[theorem]{Remark}

\numberwithin{equation}{section}

\title[Test properties of some Cohen-Macaulay modules]{Test properties of some Cohen-Macaulay modules and criteria for local rings via finite vanishing of Ext or Tor}

\author[S.~Dey]{Souvik Dey}
\address{Department of Mathematical Sciences, University of Arkansas, 850 West Dickson Street Fayetteville, Arkansas 72701 United States}
\email{souvikd@uark.edu}  
\urladdr{\url{https://orcid.org/0000-0001-8265-3301}}

\author[D.~Ghosh]{Dipankar Ghosh}
\address{Department of Mathematics, Indian Institute of Technology Kharagpur, West Bengal - 721302, India}
\email{dipankar@maths.iitkgp.ac.in, dipug23@gmail.com}
\urladdr{\url{https://orcid.org/0000-0002-3773-4003}}

\author[A.~Saha]{Aniruddha Saha}%
\address{Department of Mathematics, Indian Institute of Technology Hyderabad, Kandi, Sangareddy - 502285, Telangana, India}
\email{ma20resch11001@iith.ac.in, sahaa43@gmail.com}

\date{\today\\
	Corresponding Author: Dipankar Ghosh}
\subjclass[2020]{13D07, 13C14, 13H15, 13D05, 13H10}
\keywords{Vanishing of Ext and Tor; Auslander-Reiten Conjecture; Cohen-Macaulay modules; Hilbert-Samuel multiplicity;  Homological dimensions; Criteria for local rings}

\begin{document}

\pagenumbering{arabic}
\thispagestyle{empty}

\begin{abstract}

In this article, we show test properties, in the sense of finitely many vanishing of Ext or Tor, of CM (Cohen-Macaulay) modules whose multiplicity and number of generators (resp., type) are related by certain inequalities. We apply these test behaviour, along with other results, to characterize various kinds of local rings, including  hypersurface rings of multiplicity at most two, surprisingly requiring only finitely many vanishing of Ext or Tor involving such CM modules. As further applications, we verify the long-standing (Generalized) Auslander-Reiten Conjecture for every CM module of minimal multiplicity over a Noetherian local ring, thus vastly extending a result of Huneke-\c{S}ega-Vraciu.

\end{abstract}
\maketitle

\section{Introduction}

   Throughout this article, let $(R,\fm,k)$ be a (commutative) Noetherian local ring, and all $R$-modules are assumed to be finitely generated.

Test modules for detecting finite projective or injective dimension via eventual or single vanishing of Ext or Tor has appeared in the literature in various incarnations. Perhaps, the earliest such examples can be traced back to \cite{Ra71} and \cite{Jo85}, while more recent appearances can be found in
\cite{CDT14}, \cite{ZCGS18}, \cite{KV21} and many other papers. One aim of our paper is to study such test behaviour, in the sense of finitely many vanishing of Ext or Tor, of CM (Cohen-Macaulay) modules whose multiplicity and number of generators (resp., type) are related by certain inequalities. A large class of CM modules satisfy these inequalities. In particular, this class contains all Ulrich modules; cf.~\Cref{defn:Ulrich-mod} and \cite[Thm.~4.5 and Rmk.~4.8]{DGS}. As consequences of these test properties, we are able to characterize `regular rings (of dimension at least two)' and `hypersurface rings of multiplicity at most two' via finite vanishing of Ext or Tor using such CM modules. Moreover, we deduce from these test properties that the (Generalized) Auslander-Reiten Conjecture holds true for every CM module of minimal multiplicity.

Most of our results on test properties of modules are recorded in Theorems \ref{thm-A} and \ref{thm-B}. For the notations, which are used in these theorems, we refer the reader to the first paragraph of Section~\ref{sec:preliminaries}. Moreover, we fix the following notations for use in Theorems~\ref{thm-A}, \ref{thm-B} and \ref{thm-C}.

\begin{notation}\label{notation}
    Let $M$ and $N$ be nonzero $R$-modules. Set $d:=\dim(R)$, $r :=\dim(M)$, $s=\dim(N)$, $b_j := \beta_j^{R}(N)$, $c_j := \sum_{i=0}^d {d \choose i} \mu^{j+i}_R(N)$ and $t_j:=\sum_{i=0}^s {s \choose i} \mu^{j+i}_R(N)$ for all $j\ge 0$.
\end{notation}

\begin{customtheorem}{A}[See Theorems~\ref{thm:e-mu-pd-id} and \ref{thm:pd-e-type-test}]\label{thm-A}
    Suppose that $M$ is CM.
\begin{enumerate}[\rm(1)]
    \item 
    Assume that $e(M) < 2\mu(M) $.
        \begin{enumerate}[\rm (a)]
         \item 
         If there exists $j \ge0$ such that $\Tor^R_n(M,N)=0$ for all $n=j-r+1,\ldots,j+b_j$, then $\pd_R(N)< \infty$.
         \item 
         Suppose that either $R$ or $N$ is CM. If there exists $j\ge d+r$ such that $\Ext_R^n(M,N)=0$ for all $n= j-r+1, \ldots, j+d+c_j$, then $\id_R(N)< \infty$.
        \end{enumerate}
    \item 
    Assume that $e(M) < 2 \type (M)$. If there exists $j \ge 0$ such that $\Ext^n_R(N,M)=0$ for all $n= j+1, \ldots, j+b_j+r$, then $\pd_R(N) < \infty$.
\end{enumerate}
\end{customtheorem}

The notion of modules of minimal multiplicity is used in the next and many of the subsequent results. For a CM $R$-module $M$ of dimension $r$, it is known that $e(M) \ge \mu(\fm M) + (1-r) \mu(M)$ (cf.~\Cref{lem:min-mult-mod}). When this inequality becomes equality, $M$ is said to have minimal multiplicity, see \Cref{defn:cm-min-mult}.

\newpage
\begin{customtheorem}{B}[See \Cref{thm:Ext-pd-id-test}]\label{thm-B}
    With {\rm \Cref{notation}}, suppose that $M$ has minimal multiplicity.
    \begin{enumerate}[\rm (1)]
        \item 
        Let $e(M)>2\mu(M)$. If there exists $j\ge 0$ such that $\Ext_R^n(N,M)=0$ for all $n= j+1, \ldots, j+b_j+r$, then $\pd_R(N) < \infty$.
        \item 
        Let $e(M)<2\mu(M)$. If there exists $j\ge \depth(N)$ such that $\Ext_R^n(M,N)=0$ for all $n= j-r+1, \ldots, j+\mu_R^j(N)$, then $\id_R(N) < \infty$.
        \item \label{thm:Ext-pd-id-test-3}
        Let $e(M)>2\type(M)$. If there exists $j \ge 0$ such that $\Tor^R_n(M,N)=0$ for all $n=j-r+1,\ldots,j+b_j$, then $\pd_R(N)< \infty$.
        \item \label{thm:Ext-pd-id-test-4}
        Let $e(M)>2\type(M)$. Suppose that $N$ is CM. If there exists $j \ge 0$ such that $\Ext_R^n(M,N)=0$ for all $n= j-r+1, \dots, j+s+t_j$, then $\id_R(N)< \infty$.
    \end{enumerate}
\end{customtheorem}

A nonzero CM module of minimal multiplicity satisfying $e(M) = 2\mu(M)$ or $e(M) = 2\type(M)$ need not be a test module, see Example~\ref{exam:strict-ineq-cant-be-eq}.

As an interesting consequence of \Cref{thm-A}, our previous work in \cite{DGS} and a non-trivial result of Iyengar-Ma-Walker from \cite{IMW}, we establish several surprising characterizations of hypersurfaces of multiplicity at most two, as stated below. To the best of our knowledge, these are the first characterizations of hypersurfaces requiring only finite vanishing of Ext or Tor.

\begin{customtheorem}{C}[See \Cref{thm:char-hypersurface-e-le-2}]\label{thm-C}
Following {\rm \Cref{notation}}, the statements below are equivalent.
\begin{enumerate}[\rm (1)]
    \item 
    $R$ is a hypersurface of multiplicity at most $2$.
    \item
    $R$ admits nonzero CM modules $M$ and $N$ such that $e(M)< 2\mu(M)$, $e(N)\le 2\mu(N)$, and there exists $j\ge 0$ satisfying $\Tor^R_n(M,N)=0$ for all $n=j-r+1,\ldots,j+b_j$.
    \item 
    $R$ admits nonzero CM modules $M$ and $N$ such that $ e(M) < 2 \mu(M)$, $e(N) \le 2 \type(N) $, and there exists $j\ge 0$ satisfying $\Ext_R^n(M,N)=0$ for all $n= j-r+1, \ldots, j+s+t_j$.
    \item 
    $R$ admits nonzero CM modules $M$ and $N$ such that $e(M) < 2 \type(M) $, $ e(N) \le 2 \mu(N)$, and there exists $j \ge 0$ such that $\Ext^n_R(N,M)=0$ for all $n= j+1,\ldots, j+b_j+r$.
    \item 
    For every $n$ with $0\le n \le \dim(R)$, there exists a CM $R$-module $N$ of dimension $n$ and minimal multiplicity such that $e(N)\le 2\mu(N)$ and $\pd_R(N)<\infty$.
    \item 
    For every $n$ with $0\le n \le \dim(R)$, there exists a CM $R$-module $N$ of dimension $n$ and minimal multiplicity such that $e(N)\le 2\type(N)$ and $\id_R(N)<\infty$.
\end{enumerate}
\end{customtheorem}

In the statements of \Cref{thm-C}, replacing `$e(N)\le 2\mu(N)$' and `$e(N) \le 2 \type(N) $' by `$e(N) < 2\mu(N)$' and `$e(N) < 2 \type(N) $' respectively, we obtain similar characterizations of regular local rings, see \Cref{thm:char-reg}. Along with \Cref{thm-C}, we prove \Cref{thm:char-hyp-e-le-2-not-field}, which gives a number of characterizations of a hypersurface of multiplicity at most two that is not a field.

As further applications, we address the following long standing conjectures.

\begin{conjecture}\label{ARC}
    Let $M$ be an $R$-module.
    \begin{enumerate}[\rm (1)]
        \item (Auslander-Reiten Conjecture \cite{AR75}) If $ \Ext^n_R(M, M \oplus R) = 0 $ for all $ n \ge 1 $, then $M$ is free.
        \item (Generalized Auslander-Reiten Conjecture \cite{Wei10}) If $ \Ext^n_R(M, M \oplus R) = 0 $ for all $ n \gg 1 $, then $M$ has finite projective dimension.
    \end{enumerate}
\end{conjecture}

In \Cref{ARC}, it is immediate to observe that the statement (2) is stronger than (1). It should be mentioned that the Generalized Auslander-Reiten Conjecture (GARC) (hence also the Auslander-Reiten Conjecture, in short ARC) is known to hold true for various classes of rings and modules. For instance,  GARC holds for modules of finite CI-dimension (\cite[Cor.~4.4]{AY98} and \cite[Thm.~4.2]{AB00}), modules with finite Auslander-bound (\cite[Thm ~2.11, Def.~2.1]{Wei11}), rings with $\fm^3=0$ (\cite[Thm.~4.1]{HSV04}), all rings mentioned in \cite[Thm.~4.1 and Cor.~6.3]{AINS}, fiber-product rings (\cite[Thm.~4.5]{NS17}), Burch rings (\cite[Thms.~5.10 and 6.7]{Ta23}), Ulrich modules (\cite[Prop.~2.5.(3)]{DG23}) and Burch modules (\cite[Prop.~3.16]{DK22}). An instance where ARC is known to hold true but GARC is still open are CM normal rings (\cite[Cor.~1.3]{KOT21}). 
One of the main results of \cite{HSV04} is the validity of ARC for modules of minimal multiplicity over an Artinian local ring (\cite[Thm.~4.2]{HSV04}). From this, one can also derive that ARC holds for every MCM (maximal Cohen-Macaulay) module of minimal multiplicity over a CM local ring.
Theorems~\ref{thm:ARC-finite-pd} and \ref{thm:ARC-free-criteria}, respectively, ensure that both GARC and ARC hold for modules of minimal multiplicity over an arbitrary Noetherian local ring.  


\begin{customtheorem}{D}[See \Cref{thm:ARC-finite-pd}]\label{thm-F}
Every module of minimal multiplicity satisfies the Generalized Auslander-Reiten Conjecture.
\mycomment{
\begin{enumerate}[\rm (1)]
    \item A CM $R$-module $M$ has finite projective dimension if one of the following conditions holds.
    \begin{enumerate}[\rm (i)]
        \item 
        $ e(M) < 2 \mu(M)$, and there exists $j,l \ge 0$ such that
        \begin{enumerate}[]
            \item $\Ext_R^m(M,R)=0$ for all $m = j-\dim(M)+1,\ldots,j+d+\sum_{i=0}^d {d\choose i} \mu^{j+i}_R(R)$, where $d:=\dim(R)$, and
            \item $\Ext_R^n(M,M)=0$ for all $n = l-\dim(M)+1,\ldots,l+s+\sum_{i=0}^s {s\choose i} \mu^{l+i}_R(M)$, where $s=\dim(M)$.
        \end{enumerate}     
        \item 
        $ e(M) = 2 \mu(M) $, $M$ has minimal multiplicity, and $\Ext_R^n(M,M)=0$ for all $n\gg 0$.
        \item 
        $e(M) > 2 \mu(M) $, $M$ has minimal multiplicity, and there exists $l \ge 0$ such that $\Ext_R^n(M,M)=0$ for all $n= l+1, \ldots, l+\beta_l^{R}(M)+\dim(M)$.
    \end{enumerate}    
    \item 
    In particular, modules of minimal multiplicity satisfy the Generalized Auslander-Reiten Conjecture.
\end{enumerate}
}
\end{customtheorem}

Now, we briefly describe the structure of this paper. \Cref{sec:preliminaries} contains preliminaries on multiplicities of CM modules
and 
complexity of modules. \Cref{sec:test-modules} is devoted to proving the main results on test modules and, in particular, contains Theorems~\ref{thm-A} and \ref{thm-B}. \Cref{sec:char-complete-intersection-ring} presents various characterizations of regular rings and hypersurface rings, where we prove Theorem~\ref{thm:char-hypersurface-e-le-2}, \ref{thm:char-reg} and \ref{thm:char-hyp-e-le-2-not-field}. \Cref{sec:char-Gor-CM} provides some characterizations of Gorenstein and CM local rings. Finally, \Cref{sec:ARC-CM-MIN} applies the results from \Cref{sec:test-modules} to establish GARC for modules of minimal multiplicity.

\subsection*{Acknowledgments}
Souvik Dey was partly supported by the Charles University Research Center program No.UNCE/24/SCI/022 and a grant GA \v{C}R 23-05148S from the Czech Science Foundation. Aniruddha Saha was supported by Senior Research Fellowship (SRF) from UGC, MHRD, Govt.\,of India, and it is part of this author's Ph.D. thesis. All of the work done in this paper took place when the first author was a research scientist at the Department of Algebra of Charles University, Prague, and he is very grateful for the outstanding atmosphere fostered by the department.  

\section{Preliminaries}\label{sec:preliminaries} 

Let $M$ be an $R$-module. Let $\mu_R(M)$ and $\lambda_R(M)$ denote the minimal number of generators and the length of $M$ respectively. For every integer $n$, let $\beta_n^R(M) :=\mu(\Ext_R^n(M,k))$ and $\mu_R^n(M) := \mu(\Ext_R^n(k,M))$ be the $n$th Betti and Bass numbers of $M$ respectively. The type of $M$ is defined to be $\type_R(M):=\mu_R^t(M)$, where $t=\depth(M)$. When the base ring $R$ is clear from the context, we drop the subscript from $\mu_R(M)$, $\lambda_R(M)$ and $\type_R(M)$. We denote $(-)^\vee := \Hom_R(-,E)$, the Matlis dual, where $E$ is the injective hull of $k$.

 \begin{definition}\cite[Defn.~2.1]{GTT15}\label{defn:Ulrich-mod}
     A nonzero $R$-module $M$ is called Ulrich if it is CM and $e(M) = \mu(M)$.
 \end{definition}

We recall the following facts about multiplicity of CM modules.

\begin{lemma}\cite[Lem.~4.1]{DGS}\label{lem:min-mult-mod}
    Let $M$ be a CM $R$-module of dimension $r$. Then the following hold.
    \begin{enumerate}[\rm (1)]
        \item $e(M) \ge \mu(\fm M) + (1-r) \mu(M)$.
        \item If $\fm^2M = ({\bf x}) \fm M$ for some system of parameters ${\bf x}$ of $M$, then
        \[
            e(M) = \lambda(M/({\bf x})M) = \mu(\fm M) + (1-r) \mu(M).
        \]
        \item If the residue field $k$ is infinite, then there exists a system of parameters ${\bf x}$ of $M$ which is a reduction of $\fm$ with respect to $M$. Moreover, if $e(M) = \mu(\fm M) + (1-r) \mu(M)$ holds, then $\fm^2M = ({\bf x}) \fm M$ holds for every such reduction ${\bf x}$.
    \end{enumerate}
\end{lemma}

\begin{definition}\cite[Defn.~15]{Pu03}\label{defn:cm-min-mult}
    An $R$-module $M$ is said to have minimal multiplicity if it is CM and $e(M)=\mu(\fm M)+(1-r) \mu(M)$, where $r= \dim(M)$.
\end{definition}

\begin{remark}\label{rmk:local-duality}
Let $M$ be a module of finite length over a CM local ring $R$ of dimension $d$ which admits a canonical module $\omega$. Then, $\lambda(M)=\lambda(\Ext_R^d(M,\omega))$.
\end{remark}

\begin{proof}
Since $M$ has finite length, $M = H^0_{\fm}(M)$. Thus, by \cite[3.5.9]{BH98}, $M$ is isomorphic to the Matlis dual of $\Ext_R^d(M,\omega)$. Hence we are done by \cite[3.2.12.(b)]{BH98}.
\end{proof}

\begin{remark}\label{rmk:e-type-dagger}
Suppose that $R$ and $M$ are CM, $\dim(R)=d$, $\dim(M)=r$, and $R$ admits a canonical module $\omega$. Set $M^{\dagger}:=\Ext_R^{d-r}(M,\omega)$. Then the following hold.
\begin{enumerate}[(1)]
    \item $M^{\dagger}$ is also CM of dimension $r$, and $M^{\dagger \dagger}\cong M$.
    \item $\mu(M)=\type(M^{\dagger})$, $\type(M)=\mu(M^{\dagger})$ and $e(M)=e(M^{\dagger})$.
    \item $e(M)=\type(M)$ if and only if $M^{\dagger}$ is Ulrich if and only if $M$ is Ulrich.
\end{enumerate}
\end{remark}

\begin{proof}
    The assertion (1) is shown in \cite[3.3.10]{BH98}. In view of \cite[3.3.11.(b)]{BH98}, $\mu(M)=\type(M^{\dagger})$ and $\type(M)=\mu(M^{\dagger})$. To show the equality of multiplicities, first recall that $e(M)=\sum_{\fp\in \Spec(R), \dim(R/\fp)=r} \lambda(M_{\fp})e(R/\fp)$ by \cite[4.7.8]{BH98}. Since $\dim(M^{\dagger})=r$, similarly,  $e(M^{\dagger})=\sum_{\fp\in \Spec(R), \dim(R/\fp)=r} \lambda((M^{\dagger})_{\fp})e(R/\fp)$. Fix $\fp\in \Spec(R)$ such that $\dim(R/\fp)=r$. Then $(M^{\dagger})_{\fp} = \Ext_R^{d-r}(M,\omega)_{\fp}\cong \Ext_{R_{\fp}}^{d-r}(M_{\fp},\omega_{\fp})$. Moreover, since $\dim(R/\fp)=r=\dim(M)$, one obtains that if $\fp \in \Supp(M)$, then $\fp$ is a minimal prime of $M$, consequently, the module $M_{\fp}$ has finite length. Note that $R_{\fp}$ is CM with a canonical module $\omega_{\fp}$, and $\dim(R_{\fp})=\dim(R)-\dim(R/\fp)=d-r$. Thus, $\lambda(\Ext_{R_{\fp}}^{d-r}(M_{\fp},\omega_{\fp}))=\lambda(M_{\fp})$ by Remark~\ref{rmk:local-duality}. This concludes the proof of all the equalities in (2). Finally, (3) follows from (1), (2) and \cite[Thm.~4.5]{DGS}. 
\end{proof}

\begin{remark}\label{rmk:flat-ext-CM}
Let $(R,\fm)\to (S,\fn)$ be a flat homomorphism such that $\fn=\fm S$. Let $M$ be an $R$-module. Then, the following hold.
\begin{enumerate}[\rm(1)]
    \item $\depth_R(M)=\depth_S(M\otimes_R S)$ (see, e.g., \cite[1.2.16.(a)]{BH98}).
    \item $\dim_R(M)=\dim_S(M\otimes_R S)$ (cf.~\cite[A.11.(b)]{BH98}).
    \item \label{mh-tensor-S}$\fm^h M \otimes_R S \cong (\fm^h S)(M\otimes_R S)=\fn^h (M\otimes_R S)$ (see \cite[Lem.~5.2.5.(1)]{ddd}).
    \item $e_R(M)=e_S(M\otimes_R S)$ (see, e.g., \cite[4.3]{ddd}).
    \item $\mu_R(M)=\mu_S(M\otimes_R S)$ and $\type_R(M)=\type_S(M\otimes_R S)$.
\end{enumerate}
\end{remark}

\begin{remark}\label{rmk:flat-ext-Ulrich}
    Let $(R,\fm)\to (S,\fn)$ be a local flat homomorphism such that $\fn=\fm S$. Then, as consequences of Remark~\ref{rmk:flat-ext-CM}, the following hold.
    \begin{enumerate}[\rm(1)]
        \item
        $M$ is an $R$-module of minimal multiplicity if and only if $M\otimes_R S$ is an $S$-module of minimal multiplicity.
        \item
        $M$ is an Ulrich $R$-module if and only if $M\otimes_R S$ is an Ulrich $S$-module. (\cite[Proposition 2.2(3)]{GTT15})
    \end{enumerate}
\end{remark}

\begin{remark}\label{rmk:Ulrich-min-mult}
Any Ulrich $R$-module has minimal multiplicity. Indeed, by Remark~\ref{rmk:flat-ext-Ulrich}, we may pass to the faithfully flat extension $R[X]_{\fm[X]}$, and assume that the residue field is infinite. So, if $M$ is an Ulrich $R$-module, then in view of \cite[Prop.~2.2.(2)]{GTT15}, $\fm M=({\bf x})M$ for some system of parameters ${\bf x}$ of $M$, and hence $M$ has minimal multiplicity by \Cref{lem:min-mult-mod}. 
\end{remark}

\begin{remark}
    The results on CM modules of minimal multiplicity and the results on CM modules $M$ with $e(M)\le 2\mu(M)$ considerably strengthen those results for Ulrich modules.
\end{remark}

\begin{definition}\label{defn:ring-min-mult}
    The ring $R$ is said to have minimal multiplicity if $R$ has minimal multiplicity as an $R$-module, i.e., $R$ is CM and $e(R)=\mu(\fm)-d+1$, where $d=\dim(R)$.
\end{definition}

\begin{remark}\label{rmk:min-mult-R-M}
    Suppose that $R$ has minimal multiplicity. Then every MCM $R$-module $M$ also has minimal multiplicity. Indeed, we may assume that the residue field of $R$ is infinite. Then, there exists a maximal $R$-regular sequence $\mathbf x$ satisfying $\fm^2=(\mathbf x)\fm$. Thus $\fm^2M=(\mathbf x)\fm M$. Since $M$ is MCM, $\mathbf x$ is also a maximal $M$-regular sequence. So, by \Cref{lem:min-mult-mod}, $M$ has minimal multiplicity. 
\end{remark} 

\begin{remark}\label{rmk:superficial}
(1) It is easily observed that if $I$ is an ideal of $R$ such that $R/I$ is CM, then $R/I$ has minimal multiplicity as a ring if and only if it has minimal multiplicity as an $R$-module.

(2) Let $R$ be CM of dimension $d$ with infinite residue field. Then, by \cite[preceding paragraph of Lem.~2.4]{Gho19}, there exists an $R$-regular sequence $x_1,\ldots,x_d$, which is a part of a minimal generating set of $\fm$, and $e(R)=e(R/(x_1,\ldots,x_n))$ for every $0\le n \le d$. It follows that for such a sequence $x_1,\dots,x_n$, the ring $R$ has minimal multiplicity if and only if so does $R/(x_1,\dots,x_n)$. Indeed, fix such an $n$ and put $X:=R/(x_1,\ldots,x_n)$. Then, $(1-\dim(X))\mu(X)=1-d +n$, and $\mu(\fm X)=\mu(\fm/(x_1,\dots,x_n))=\mu(\fm)-n$. Hence, $\mu(\fm X)+(1-\dim(X))\mu(X)=\mu(\fm)-n+1-d+n=\mu(\fm)-d+1$. Since $e(X)=e(R)$, we are now done by Definitions~\ref{defn:cm-min-mult} and \ref{defn:ring-min-mult}. 
\end{remark}

Here we recall the notion of complexity of a module.

\begin{definition}\label{defn:cx}
    The complexity of an $R$ module $M$, denoted $\cx_R(M)$, is defined to be the smallest non-negative integer $b$ such that $\beta_n^R (M)\le \alpha n^{b-1}$ for some real number $\alpha>0$ and for all $n \gg 0$. If no such $b$ exists, then $\cx_R(M) :=\infty$.
    
    Similarly, replacing $\beta_n^R(M)$ by $\mu^n_R(M)$, one obtains the notion of injective complexity of $M$, denoted by $\injcx_R(M)$.
\end{definition}

The residue field has extremal complexity. Moreover, complexity  of the residue field characterize complete intersection local rings.

\begin{proposition}\label{prop:cx-curv-facts}{\ }
	\begin{enumerate}[\rm (1)]
		\item \label{cx-curv-M-k}\cite[Prop.~2]{Avr96} For every $R$-module $M$,
		\begin{enumerate}[\rm (a)]
			\item $\cx_R(M) \le \cx_R(k)$ and $\injcx_R(M) \le \injcx_R(k) =\cx_R(k)$.
			
		\end{enumerate}
		\item \label{char-CI-via-cx-curv}\cite[Thm.~3]{Avr96} The following statements are equivalent:
		\begin{enumerate}[\rm (a)]
			\item $R$ is complete intersection $($resp., of codimension $c$$)$.
			\item $\cx_R(k)<\infty$ $($resp., $\cx_R(k)=c<\infty$$)$.
			
		\end{enumerate}
		\item \label{CI-ring-finite-cx-curv}It follows from $(1)$ and $(2)$ that if $R$ is complete intersection, then for every $R$-module $M$, $\cx_R(M) < \infty$, $\injcx_R(M) < \infty$.
	\end{enumerate}
\end{proposition}

We recall the following results from our previous work \cite{DGS} which will be used several times in the upcoming sections.

\begin{proposition}\cite[cf.~Cor.~5.5]{DGS}\label{DGS.5.5}
Let $M$ be a nonzero CM $R$-module.
\begin{enumerate}[\rm(1)]
    \item If $e(M)\le 2\mu(M) $, then $\cx_R(k)\le 1+\cx_R(M)$, i.e., $\cx_R(k)=\cx_R(M)$ or $1+\cx_R(M)$.
    \item If $e(M)<2\mu(M) $, then $\cx_R(k) = \cx_R(M)$ .
    \item If $e(M)\le 2\type(M) $, then $\cx_R(k)\le 1+\injcx_R(M)$, i.e., $\cx_R(k)=\injcx_R(M)$ or $1+\injcx_R(M)$.
    \item If $e(M)<2\type(M) $, then $\cx_R(k)=\injcx_R(M)$ .
\end{enumerate}
\end{proposition}

\begin{proposition}\cite[cf.~Cor.~6.6]{DGS}\label{DGS.6.6}
Let $M$ be a nonzero $R$-module of minimal multiplicity.
\begin{enumerate}[\rm(1)]
    \item
    If $e(M)\ge 2\mu(M)$, then $\cx_R(k)\le 1+\injcx_R(M)$, i.e., $\cx_R(k)\in\{\injcx_R(M),1+\injcx_R(M)\}$.
    \item 
    If $e(M)>2\mu(M)$, then $\cx_R(k)=\injcx_R(M)$.
    \item
    If $e(M)\ge 2\type(M)$, then $\cx_R(k)\le 1+\cx_R(M)$, i.e., $\cx_R(k)\in\{\cx_R(M),1+\cx_R(M)\}$.
    \item 
    If $e(M)>2\type(M)$, then $\cx_R(k)=\cx_R(M)$.
\end{enumerate}
\end{proposition}

\section{Vanishing of Ext and Tor for some CM modules}\label{sec:test-modules}

Here we mainly show that vanishing of (co)homologies involving certain CM module detect finite homological dimensions, namely projective and injective dimensions, of the other module. We use the following terminologies.

\begin{definition}\label{defn:test-modules}
    An $R$-module $M$ is said to be:
    \begin{enumerate}[\rm (1)]
        \item 
        Tor-pd-test if for every $R$-module $N$, whenever $\Tor_{\gg 0}^R (M,N)=0$, one has that $\pd_R(N)< \infty$. It was introduced as test module in \cite[Defn.~1.1]{CDT14}.
        \item 
        Ext-pd-test if for every $R$-module $N$, whenever $\Ext^{\gg 0}_R (N,M)=0$, one has that $\pd_R(N)< \infty$.
        \item 
        Ext-id-test if for every $R$-module $N$, whenever $\Ext^{\gg 0}_R (M,N)=0$, one has that $\id_R(N)< \infty$.
    \end{enumerate}
\end{definition}

The following proposition ensures that every nonzero module $M$ with $ \lambda(M) < 2 \mu(M) $ is Tor-pd-test as well as Ext-id-test when testing for injective dimension of modules of finite length.

\begin{proposition}\label{prop:lambda-mu-id-pd}
    Let $X$ and $M$ be nonzero $R$-modules such that $\lambda(M) < 2\mu(M) $.
    \begin{enumerate}[\rm (1)]
        \item 
        If there exists $j\ge 0$ such that $\Tor^R_n(M,X)=0$ for all $ n = j+1,\ldots,j+\beta_j^{R}(X) $, then $ \pd_R(X) < \infty$.
        \item 
        If $\lambda(X) < \infty $ and there exists $j \ge 0$ such that $\Ext^n_R(M,X)=0$ for all $n= j+1, \dots, j+\mu^j_R(X)$, then $\id_R(X)< \infty$.
    \end{enumerate}
\end{proposition}

\begin{proof}
    (1) If possible, assume that $\pd_R(X) = \infty$. Then $\beta_m^R(X) \neq 0$ for all $m\ge 0$. Since $\lambda(M) < 2 \mu(M)$, it follows that $\frac{\lambda(M)}{\mu(M)}-1< 1$. Hence, by \cite[Lem.~5.1.(1)]{DGS}, one obtains that $\beta_{n}^R(X) < \beta_{n-1}^R(X)$, i.e., $\beta_{n}^R(X) \le \beta_{n-1}^R(X)-1$ for all $n= j+1, \dots, j+\beta_j^R(X)$. Hence, setting $s:= j+\beta_j^R(X)$,
	\[
	   \beta_{s}^R(X)\le \beta_{s-(s-j)}^R(X)-(s-j) = \beta_{j}^R(X)-(s-j) = \beta_{j}^R(X)-\beta_{j}^R(X) = 0.
	\]
	Thus $\beta_{s}^R(X) = 0$, which is a contradiction. Therefore $\pd_R(X)< \infty$.
	
	(2) Since $\lambda(X) < \infty$, one has that $\Ext_R^n(M,X)^\vee \cong \Tor^R_n(M,X^\vee)$ and $\mu_R^n(X)=\beta_n^R(X^\vee)$ for all $n \ge 0$. So, from the given vanishing condition, $\Tor^R_n(M,X^\vee)=0$ for all $n= j+1, \dots, j+\beta_j^R(X^\vee)$. Therefore, by (1), $\pd_R(X^\vee) < \infty$, and hence $\id_R(X) < \infty$.
\end{proof}

Every nonzero $R$-module $M$ with $ \lambda(M) < 2 \type(M) $ is Ext-pd-test.

\begin{proposition}\label{prop:lambda-bass-pd}
    Let $M$ be a nonzero $R$-module such that $\lambda(M) < 2 \type(M)$. Let $X$ be an $R$-module, and there exists $j \ge 0$ such that $\Ext^n_R(X,M)=0$ for all $n= j+1, \dots, j+\beta_j^R(X)$. Then $\pd_R(X)< \infty$.
\end{proposition}
	
\begin{proof}
	 If possible, assume that $\pd_R(X) = \infty$. Then $\beta_m^R(X) \neq 0$ for all $m\ge 0$. Since $\lambda(M) < 2 \type(M)$, one has that $\frac{\lambda(M)}{\type(M)}-1 < 1$. Therefore, by \cite[Lem.~5.1.(2)]{DGS}, one obtains that $\beta_{n}^R(X) < \beta_{n-1}^R(X)$, i.e., $\beta_{n}^R(X) \le \beta_{n-1}^R(X)-1$ for all $n= j+1, \dots, j+\beta_j^R(X)$. Hence, setting $s:= j+\beta_j^R(X)$,
	\[
	\beta_{s}^R(X)\le \beta_{s-(s-j)}^R(X)-(s-j) = \beta_{j}^R(X)-(s-j) = \beta_{j}^R(X)-\beta_{j}^R(X) = 0.
	\]
	Thus $\beta_{s}^R(X) = 0$, which is a contradiction. Therefore $\pd_R(X)< \infty$.
\end{proof}

Next we analyze the testing properties of a nonzero $R$-module $M$ satisfying $ \fm^2 M = 0 $. The reader may compare \Cref{prop:pd-id-test-m2M-zero}.(2) with \Cref{prop:lambda-mu-id-pd}.(2).

\begin{proposition}\label{prop:pd-id-test-m2M-zero}
	Let $M$ and $X$ be nonzero $R$-modules such that $\fm^2 M = 0$.  
 \begin{enumerate}[\rm (1)]
     \item 
     If $ \lambda(M) > 2 \mu(M) $ and there exists $j\ge 0$ such that $\Ext_R^n(X,M)=0$ for all $n= j+1, \ldots, j+\beta_j^{R}(X)$, then $\pd_R(X) < \infty$.
     \item 
     If $ \lambda(M) < 2 \mu(M) $ and there exists $j\ge \depth(X)$ such that $\Ext^n_R(M,X)=0$ for all $n= j+1, \dots, j+\mu^j_R(X)$, then $\id_R(X)< \infty$.
     \item 
     If $ \lambda(M) > 2 \type(M) $ and there exists $j\ge 0$ such that $\Tor^R_n(X,M)=0$ for all $n= j+1, \ldots, j+\beta_j^{R}(X)$, then $\pd_R(X) < \infty$.
     \item If $\lambda(M)>2\type(M)$, $X$ has finite length, and there exists $j\ge 0$ such that $\Ext_R^n(M,X)=0$ for all $n= j+1, \dots, j+\mu^j_R(X)$, then $\id_R(X)< \infty$.
 \end{enumerate} 
\end{proposition}

\begin{proof}
(1)	If possible, assume that $\pd_R(X) = \infty$. Then $\beta_m^R(X) \neq 0$ for all $m\ge 0$. Since $\lambda(M) > 2 \mu(M)$, it follows that $\frac{\mu(M)}{\lambda(M)-\mu(M)}< 1$. So, by \cite[Lem.~6.1.(1)]{DGS}, one obtains that $\beta_{n}^R(X) < \beta_{n-1}^R(X)$, i.e., $\beta_{n}^R(X) \le \beta_{n-1}^R(X)-1$ for all $n= j+1, \dots, j+\beta_j^R(X)$. Hence, setting $s:= j+\beta_j^R(X)$,
\[
   \beta_{s}^R(X)\le \beta_{s-(s-j)}^R(X)-(s-j) = \beta_{j}^R(X)-\beta_{j}^R(X) = 0.
\]
Thus $\beta_{s}^R(X) = 0$, which is a contradiction. Therefore $\pd_R(X)< \infty$.

(2) If possible, assume that $\id_R(X)=\infty$. Then, by \cite[II.~Thm.~2]{Rob76}, $\mu^m_R(X)\neq 0$ for all $m\ge \depth(X)$. 
Since $\lambda(M) < 2 \mu(M)$, it follows that $\frac{\lambda(M)}{\mu(M)}-1< 1$. So, in view of \cite[Lem.~6.1.(2)]{DGS}, $\mu^{n}_R(X) < \mu^{n-1}_R(X)$, i.e., $\mu^{n}_R(X) \le \mu^{n-1}_R(X)-1$ for all $n= j+1, \dots, j+\mu^j_R(X)$. Hence, setting $s:= j+\mu^j_R(X)$,
\[
   \mu^{s}_R(X)\le \mu^{s-(s-j)}_R(X)-(s-j) = \mu^{j}_R(X)-\mu^{j}_R(X) = 0.
\]
Thus $\mu^{s}_R(X) = 0$, which is a contradiction as $s= j+\mu^j_R(X)\ge j\ge \depth(X)$. So $\id_R(X)< \infty$. 

(3) Note that $\fm^2(M^{\vee})=0$, $\lambda(M)=\lambda(M^{\vee})$ and $\type(M)=\mu(M^{\vee})$. Hence, since $\Tor^R_i(X,M)^{\vee}\cong \Ext_R^{i}(X,M^{\vee})$ for all $i\ge 0$, the result follows by using (1) for $M^{\vee}$ and $X$.

(4) Since $\lambda(X) < \infty$, note that $\Ext^n_R(M,X)^{\vee} \cong \Tor^R_n(M, X^{\vee})$ for all $n \ge 0$. This yields that $\id_R(X)<\infty$ \iff $\pd_R(X^{\vee})<\infty$. Moreover, $\mu^n_R(X)=\beta_n^R(X^{\vee})$ for all $n\ge 0$. Therefore the desired result follows by using (3) for $M$ and $X^{\vee}$.
\end{proof}

We need the following elementary lemma about vanishing of Ext and Tor.

\begin{lemma}\cite[Lem.~2.3]{DG23}\label{lem:cons-vanishing-Tor}
	Consider two integers $m$ and $n$. Let ${\bf x} = x_1,\ldots,x_t$ be an $M$-regular sequence of length $t$. Then the following hold true.
	\begin{enumerate}[\rm (1)]
		\item If $\Tor^R_i(M,N)=0$ for $n\le i\le m+n$, then $\Tor^R_i(M/{\bf x} M,N)=0$ for $n+t\le i\le m+n$.
		\item   If $\Ext^i_R(M,N)=0$ for $n\le i \le m+n$, then  $\Ext^i_R(M/{\bf x} M,N) = 0$ for all $n+t\le i\le m+n$.		
		\item  If $\Ext^i_R(N,M)=0$ for all $n\le i\le m+n$, then $\Ext^i_R(N,M/{\bf x} M)=0$ for all $n\le i\le m+n-t$.   
	\end{enumerate} 
\end{lemma}

The following lemma gives a relation between Bass numbers of a module $M$ and that of $M/{\bf x}M$, where ${\bf x}$ is an $M$-regular sequence.

\begin{lemma}\label{lem:mu-C}
   Let $M$ be an $R$-module, and ${\bf x}=x_1,\ldots, x_s$ be an $M$-regular sequence of length $s$. Then, for each $j\ge 0$, $\mu^j_R\big(M/({\bf x})M \big)= \sum_{i=0}^s {s \choose i} \mu^{j+i}_R(M)$, where ${ s \choose i } = s!/i!(s-i)!$.
\end{lemma}

\begin{proof}
    When $s=0$, there is nothing to prove. So we consider the base case $s=1$. There is an exact sequence $0 \to M \stackrel{x_1} \longrightarrow M \to M/{x_1}M \to 0$, which induces another exact sequence $0 \to \Ext^j_R(k,M) \to \Ext^j_R(k,M/x_1M) \to \Ext^{j+1}_R(k,M) \to 0$. This implies that $\mu^j_R(M/x_1M)=\mu^j_R(M)+\mu^{j+1}_R(M)=\sum_{i=0}^1 { 1 \choose i } \mu^{j+i}_R(M)$. By the induction hypothesis, we may assume that the result holds for $s-1$, i.e., $\mu^j_R\big(M/({\bf y})M \big)= \sum_{i=0}^{s-1} {s-1 \choose i} \mu^{j+i}_R(M)$, where ${\bf y} := x_1,\ldots, x_{s-1} $. Note that $x_s$ is a regular element on $M' := M/{\bf y}M$. So, by the base case, $ \mu^j_R \big( M/({\bf x})M \big)= \mu^j_R(M'/x_s M') = \mu^j_R(M')+\mu^{j+1}_R(M')= \sum_{i=0}^{s-1} {s-1 \choose i} \mu^{j+i}_R(M) + \sum_{i=0}^{s-1} {s-1 \choose i} \mu^{j+i+1}_R(M) = \mu^j_R(M) + \left( \sum_{i=1}^{s-1} \left\{{s-1 \choose i} + {s-1 \choose i-1}\right\} \mu^{j+i}_R(M) \right) + \mu^{j+s}_R(M) = \sum_{i=0}^s {s \choose i} \mu^{j+i}_R(M)$.
\end{proof}

Now we are in a position to prove that every CM module $M$ such that $e(M)< 2 \mu(M)$ is Tor-pd-test for arbitrary modules as well as it is Ext-id-test for CM modules. We remark that such a module $M$ is a $c$-Ulrich module, for every $c<2$, in the sense of \cite{CCLTY}. We also note that our result extends similar test properties already known for Ulrich modules \cite{GP19,DG23}. 

\begin{theorem}\label{thm:e-mu-pd-id}
    Let $M$ be a CM $R$-module such that $e(M) < 2\mu(M) $. Let $N$ be an $R$-module.
    \begin{enumerate}[\rm (1)]
		\item 
        If there exists $j \ge0$ such that $\Tor^R_n(M,N)=0$ for all $n=j-\dim(M)+1,\ldots,j+\beta_j^{R}(N)$, then $\pd_R(N)< \infty$. In particular, if $\Tor^R_{\gg 0}(M,N)=0$, then $\pd_R(N)< \infty$.
        \item 
        Suppose that $N$ is CM of dimension $s$. If there exists $j\ge 0$ such that $\Ext_R^n(M,N)=0$ for all $n= j-\dim(M)+1, \ldots, j+s+\sum_{i=0}^s {s \choose i} \mu^{j+i}_R(N)$, then $\id_R(N)< \infty$. In particular, if $\Ext_R^{\gg 0}(M,N)=0$, then $\id_R(N)< \infty$.
        \item 
        Suppose that $R$ is CM of dimension $d$. If there exists $j\ge \dim(M)+d$ such that $\Ext_R^n(M,N)=0$ for all $n= j-\dim(M)+1, \ldots, j+d+\sum_{i=0}^d {d \choose i} \mu^{j+i}_R(N)$, then $\id_R(N)< \infty$. In particular, if $\Ext_R^{\gg 0}(M,N)=0$, then $\id_R(N)< \infty$.
	\end{enumerate}
\end{theorem}

\begin{proof}
	Without loss of generality, we may assume that $k$ is infinite. Then, by a similar argument as in \cite[Rmk.~2.4]{DG23}, there exists a system of parameters ${\bf x}$ of $M$ such that $e(M) = \lambda(M/{\bf x}M)$. Since $M$ is CM, the sequence ${\bf x}$ is $M$-regular. Set $L:=M/{\bf x}M$. Then $e(M) = \lambda(L)$ and $\mu(M) = \mu(L)$. Thus $\lambda(L) < 2 \mu(L)$.
	
    (1) If there exists $j \ge 0$ such that $\Tor^R_n(M,N)=0$ for all $n=j-\dim(M)+1,\ldots,j+\beta_j^{R}(N)$, then by Lemma~\ref{lem:cons-vanishing-Tor}.(1), $\Tor^R_n(L, N) = 0$ for all $n=j+1,\ldots,j+\beta_j^{R}(N)$, and hence $\pd_R(N)< \infty$ by Proposition~\ref{prop:lambda-mu-id-pd}.(1). The particular case follows from the general one.
 %

    (2) Since $N$ is CM of dimension $s$, there exists an $N$-regular sequence ${\bf y} = y_1,\ldots,y_s$ such that $\lambda(N/{\bf y}N)<\infty$. Then, in view of Lemmas~\ref{lem:cons-vanishing-Tor}.(3) and \ref{lem:mu-C}, from the given vanishing condition, one obtains that $\Ext_R^n(M,N/{\bf y}N)=0$ for all $n= j-\dim(M)+1, \ldots, j+\mu^j_R \big( N/{\bf y}N \big)$, and hence $\Ext_R^n(L,N/{\bf y}N)=0$ for all $n= j+1, \ldots, j+\mu^j_R \big( N/{\bf y}N \big)$ (by Lemma~\ref{lem:cons-vanishing-Tor}.(2)). Therefore, since $\lambda(L) < 2 \mu(L)$ and $\lambda(N/{\bf y}N)<\infty$, \Cref{prop:lambda-mu-id-pd}.(2) yields that $ \id_R(N/{\bf y}N) < \infty $, which implies that $ \id_R(N) < \infty $.

    (3) We may pass to completion, and assume that $R$ is complete. Hence $R$ admits a canonical module. Take a minimal MCM approximation $0\to Y\to N'\to N\to 0$ of $N$ as guaranteed by \cite[11.13, 11.14 and 11.17]{LW}, where $N'$ is MCM and  $Y$ has finite injective dimension. It follows that $\Ext^n_R(X,N')\cong \Ext^n_R(X,N)$ for every $R$-module $X$ and for all $n>d$, moreover, when $X=k$, $\Ext^d_R(k,N)\cong \Ext^d_R(k,N')$ also holds by the discussion made in \cite[p.183]{LW}. So $\mu^{j+i}_R(N')=\mu^{j+i}_R(N)$ for all $j\ge d$ and $i\ge 0$. Also, in our hypothesis, $n\ge j-\dim(M)+1\ge d+1$. Thus, we may replace $N$ by $N'$, and assume that $N$ is MCM. Hence the result follows from (2).
\end{proof}

Next we show that every nonzero CM module $M$ such that $e(M) < 2 \type(M)$ is Ext-pd-test.

\begin{theorem}\label{thm:pd-e-type-test}
    Let $M$ be a nonzero CM $R$-module such that $e(M) < 2 \type (M)$. Let $N$ be an $R$-module. If there exists $j \ge 0$ such that $\Ext^n_R(N,M)=0$ for all $n= j+1, \ldots, j+\beta_j^{R}(N)+\dim(M)$, then $\pd_R(N) < \infty$. In particular, if $\Ext_R^{\gg 0}(N,M)=0$, then $\pd_R(N)<\infty$.
\end{theorem}

\begin{proof}
     As in the proof of Theorem~\ref{thm:e-mu-pd-id}, set $L:=M/{\bf x}M$, where ${\bf x}$ is an $M$-regular sequence such that $e(M) = \lambda(L)$. Then, by \cite[Lem.~1.2.4.]{BH98}, $\type(M)= \type(L)$. So $0< \lambda(L) < 2 \type(L)$. Thus, if there exists $j \ge 0$ such that $\Ext^n_R(N,M)=0$ for all $n= j+1, \ldots, j+\beta_j^{R}(N)+\dim(M)$, then by Lemma~\ref{lem:cons-vanishing-Tor}.(3), $\Ext^n_R(N,L)=0$ for all $n= j+1, \ldots, j+\beta_j^{R}(N)$, and hence $\pd_R(N)<\infty$ by Proposition~\ref{prop:lambda-bass-pd}.
\end{proof}

We now analyze the testing properties of CM modules of minimal multiplicity. The second part of the next theorem shows that one can drop the CM hypothesis on $N$ in Theorem~\ref{thm:e-mu-pd-id}.(2) at the expense of assuming $M$ has minimal multiplicity.

\begin{theorem}\label{thm:Ext-pd-id-test}
    Let $M$ and $N$ be nonzero $R$-modules such that $M$ has minimal multiplicity.
    \begin{enumerate}[\rm (1)]
        \item 
        Let $e(M)>2\mu(M)$. If there exists $j\ge 0$ such that $\Ext_R^n(N,M)=0$ for all $n= j+1, \ldots, j+\beta_j^{R}(N)+\dim(M)$, then $\pd_R(N) < \infty$.
        \item 
        Let $e(M)<2\mu(M)$. If there exists $j\ge \depth(N)$ such that $\Ext_R^n(M,N)=0$ for all $n= j-\dim(M)+1, \ldots, j+\mu_R^j(N)$, then $\id_R(N) < \infty$.
        \item \label{thm:Ext-pd-id-test-3}
        Let $e(M)>2\type(M)$. If there exists $j \ge 0$ such that $\Tor^R_n(M,N)=0$ for all $n=j-\dim(M)+1,\ldots,j+\beta_j^{R}(N)$, then $\pd_R(N)< \infty$.
        \item \label{thm:Ext-pd-id-test-4}
        Let $e(M)>2\type(M)$. If $N$ is CM of dimension $s$, and there exists $j \ge 0$ such that $\Ext_R^n(M,N)=0$ for all $n= j-\dim(M)+1, \dots, j+s+\sum_{i=0}^s {s \choose i} \mu^{j+i}_R(N)$, then $\id_R(N)< \infty$.
        \item Suppose that $R$ is CM of dimension $d$, and $e(M)>2\type(M)$. If there exists $j \ge \dim (M)+d$ such that $\Ext_R^n(M,N)=0$ for all $n= j-\dim(M)+1, \dots, j+d+\sum_{i=0}^d {d \choose i} \mu^{j+i}_R(N)$, then $\id_R(N)< \infty$.
    \end{enumerate}
\end{theorem}

\begin{proof}
Without loss of generality, we may assume that $k$ is infinite. Then, in view of \Cref{lem:min-mult-mod}, there exists a system of parameters ${\bf x}$ of $M$ such that $\fm^2M = ({\bf x}) \fm M$ and $e(M) = \lambda(M/{\bf x}M)$. Since $M$ is CM, ${\bf x}$ is also an $M$-regular sequence. Setting $L:=M/{\bf x}M$, we have that $\fm^2 L=0$, $e(M) = \lambda(L)$, $\mu(M) = \mu(L)$ and $\type(M) = \type(L)$. 

(1) Under the vanishing condition of (1), by Lemma~\ref{lem:cons-vanishing-Tor}.(3), $\Ext_R^n(N,L)=0$ for all $n= j+1, \ldots, j+\beta_j^{R}(N)$, and hence $\pd_R(N) < \infty$ by \Cref{prop:pd-id-test-m2M-zero}.(1).

(2) Under the vanishing condition of (2), by Lemma~\ref{lem:cons-vanishing-Tor}.(2), $\Ext_R^n(L,N)=0$ for all $n= j+1, \ldots, j+\mu_R^j(N)$, and hence $\id_R(N) < \infty$ by \Cref{prop:pd-id-test-m2M-zero}.(2).

(3) Under the vanishing condition of (3), by Lemma~\ref{lem:cons-vanishing-Tor}.(1), $\Tor^R_n(L,N)=0$ for all $n= j+1, \ldots, j+\beta_j^{R}(N)$, and hence \Cref{prop:pd-id-test-m2M-zero}.(3) yields that $\pd_R(N) < \infty$.

(4) Since $N$ is CM, there exists an $N$-regular sequence ${\bf y} = y_1,\ldots,y_s$ such that $\lambda(N/{\bf y}N)<\infty$. Then, in view of Lemmas~\ref{lem:cons-vanishing-Tor}.(3) and \ref{lem:mu-C}, from the given vanishing condition, one obtains that $\Ext_R^n(M,N/{\bf y}N)=0$ for all $n= j-\dim(M)+1, \ldots, j+\mu^j_R \big( N/{\bf y}N \big)$, and hence $\Ext_R^n(L,N/{\bf y}N)=0$ for all $n= j+1, \ldots, j+\mu^j_R \big( N/{\bf y}N \big)$ (by Lemma~\ref{lem:cons-vanishing-Tor}.(2)). Therefore, since $\lambda(L) > 2 \type(L)$ and $\lambda(N/{\bf y}N)<\infty$, \Cref{prop:pd-id-test-m2M-zero}.(4) yields that $ \id_R(N/{\bf y}N) < \infty $, which implies that $ \id_R(N) < \infty $.

(5) We may pass to completion, and assume that $R$ is complete. Hence $R$ admits a canonical module. Take a minimal MCM approximation $0\to Y\to N'\to N\to 0$ of $N$ as guaranteed by \cite[11.13, 11.14 and 11.17]{LW}, where $N'$ is MCM and  $Y$ has finite injective dimension. It follows that $\Ext^n_R(X,N')\cong \Ext^n_R(X,N)$ for every $R$-module $X$ and for all $n>d$, moreover, when $X=k$, $\Ext^d_R(k,N)\cong \Ext^d_R(k,N')$ also holds by the discussion made in \cite[p.183]{LW}. So $\mu^{j+i}_R(N')=\mu^{j+i}_R(N)$ for all $j\ge d$ and $i\ge 0$. Also, in our hypothesis, $n\ge j-\dim(M)+1\ge d+1$. Thus, we may replace $N$ by $N'$, and assume that $N$ is MCM. Hence the result follows from (4).
\end{proof}

The conditions `$e(M) < 2 \mu(M)$', `$e(M) < 2 \type(M)$', `$e(M) > 2 \mu(M)$' and `$e(M) > 2 \type(M)$' in Theorems~\ref{thm:e-mu-pd-id}, \ref{thm:pd-e-type-test} and \ref{thm:Ext-pd-id-test} cannot be replaced by `$e(M) \le 2 \mu(M)$', `$e(M) \le 2 \type(M)$', `$e(M) \ge 2 \mu(M)$' and `$e(M) \ge 2 \type(M)$' respectively. Indeed, there are examples of CM modules $M$ (of minimal multiplicity) with $e(M) = 2 \mu(M)=2\type(M)$  which are neither Tor-pd-test, nor Ext-id-test, nor Ext-pd-test modules.

\begin{example}\label{exam:strict-ineq-cant-be-eq}
    Consider a hypersurface $R$ of multiplicity $2$ (e.g., $R=k[x]/(x^2)$). Then, the $R$-module $M:=R$ is CM of minimal multiplicity such that $e(M)=2=2\mu(M)=2\type(M)$. Note that $\Tor_i^R(M,k)=0$ and $\Ext_R^i(M,k)=0$ for all $i >0$, and $\Ext_R^j(k,M)=0$ for all $j>\dim(R)$. However, $\pd_R(k)=\infty=\id_R(k)$ as $R$ is not regular.  
    
\end{example}

\section{Characterizations of complete intersection local rings}\label{sec:char-complete-intersection-ring}

In this section, we provide a number of characterizations of complete intersection local rings in terms of vanishing of Ext and Tor involving CM modules of certain multiplicities. Note that the class of complete intersection local rings contains regular local rings and hypersurfaces. We start with characterizations of hypersurface of multiplicity at most $2$. Recall that $R$ is called a hypersurface if $R$ is a complete intersection of codimension at most $1$, equivalently, if $R$ is CM and $\mu(\fm)-\dim(R)\le 1$.

\begin{theorem}\label{thm:char-hypersurface-e-le-2}
The following statements are equivalent.
\begin{enumerate}[\rm (1)]
    \item 
    $R$ is a hypersurface of multiplicity at most $2$.
    \item 
    For every $n$ with $0\le n \le \dim(R)$, there exists a CM $R$-module $N$ of dimension $n$ and minimal multiplicity such that $e(N)\le 2\mu(N)$ and $\pd_R(N)<\infty$.
    \item 
    For every $n$ with $0\le n \le \dim(R)$, there exists a CM $R$-module $N$ of dimension $n$ and minimal multiplicity such that $e(N)\le 2\type(N)$ and $\id_R(N)<\infty$.
    \item 
    $R$ admits a nonzero CM module $N$ such that $e(N)\le 2\mu(N)$ and $\pd_R(N)<\infty$.
    \item 
    $R$ admits a non-zero CM module $N$ such that $e(N)\le 2\type(N)$ and $\id_R(N)<\infty$.
    \item
    $R$ admits nonzero CM modules $M$ and $N$ such that $e(M)< 2\mu(M)$, $e(N)\le 2\mu(N)$, and there exists $j\ge 0$ satisfying $\Tor^R_n(M,N)=0$ for all $n=j-\dim(M)+1,\ldots,j+\beta_j^{R}(N)$.
    \item 
    $R$ admits nonzero CM modules $M$ and $N$ such that $ e(M) < 2 \mu(M)$, $e(N) \le 2 \type(N) $, and there exists $j\ge 0$ satisfying $\Ext_R^n(M,N)=0$ for all $n= j-\dim(M)+1, \ldots, j+s+\sum_{i=0}^s {s \choose i} \mu^{j+i}_R(N)$, where $s=\dim(N)$.
    \item 
    $R$ admits nonzero CM modules $M$ and $N$ such that $e(M) < 2 \type(M) $, $ e(N) \le 2 \mu(N)$, and there exists $j \ge 0$ such that $\Ext^n_R(N,M)=0$ for all $n= j+1,\ldots, j+\beta_j^{R}(N)+\dim(M)$.
\end{enumerate}
Moreover, in each of the conditions {\rm (6)--(8)} above, it follows that $\pd_R(N)<\infty$.     
\end{theorem}

\begin{proof}
Passing to the faithfully flat extension $R[X]_{\fm[X]}$, we may assume that $R$ has infinite residue field.

(1) $\Longrightarrow$ (2) and (1) $\Longrightarrow$ (3): Set $d:=\dim(R)$. Fix $n$ with $0\le n \le d$. Denote $m:=d-n$. Since $e(R)\le 2$, in particular, $R$ has minimal multiplicity. Moreover, by Remark~\ref{rmk:superficial}, there exists an $R$-regular sequence ${\bf x} = x_1,\ldots,x_m$ such that $e(R) = e(R/(\bf x))$, and $R/(\bf x)$ has minimal multiplicity as an $R$-module. Set $N:=R/(\mathbf x)$. Then $N$ is a CM $R$-module of dimension $n$ and minimal multiplicity. Also $\pd_R(N)<\infty$ and $e(N) = e(R) \le 2=2\mu(N)$. Since $R$ is in particular Gorenstein, $\id_R(N)<\infty$ and $\type(N)=\type(R)=1$. Hence $e(N) \le 2=2\type(N)$. So the implications follow.

(2) $\Longrightarrow$ (4) and (3) $\Longrightarrow$ (5): The implications are trivial.

(4) $\Longrightarrow$ (6), (5) $\Longrightarrow$ (7) and (4) $\Longrightarrow$ (8): These implications are obvious by setting $M:=k$.

(6) $\Longrightarrow$ (1): By Theorem~\ref{thm:e-mu-pd-id}.(1), $\pd_R(N)<\infty$, i.e., $\cx_R(N) = 0$. Hence, since $e(N)\le 2\mu(N)$, \Cref{DGS.5.5}.(1) yields that $\cx_R(k)\le 1$. Thus, $R$ is a hypersurface by Proposition~\ref{prop:cx-curv-facts}.(2). There exists a maximal $N$-regular sequence $\mathbf x$ such that $e(N)=\lambda(N/\mathbf xN)$. Put $L:=N/\mathbf x N$, which has finite length, and $\pd_R(L)<\infty$. Moreover, by hypothesis, $\lambda(L)=e(N)\le 2\mu(N)=2\mu(L)$. By \cite[Thm.~3.8]{IMW}, $e(R) \mu(L)\le \lambda(L)$. Thus, $e(R)\mu(L)\le 2\mu(L)$, and hence $e(R)\le 2$.

(7) $\Longrightarrow$ (1): In view of Theorem~\ref{thm:e-mu-pd-id}.(2), $\id_R(N) < \infty$, i.e., $\injcx_R(N) = 0$. Therefore, since $ e(N) \le 2 \type(N) $, \Cref{DGS.5.5}.(3) yields that $\cx_R(k) \le 1$, and hence $R$ is a hypersurface by Proposition~\ref{prop:cx-curv-facts}.(2). It follows that $\pd_R(N)<\infty$ as $\id_R(N) < \infty$. There exists a maximal $N$-regular sequence $\mathbf x$ such that $e(N)=\lambda(N/\mathbf xN)$. Put $L:=N/\mathbf x N$, which has finite length. Then, $\type(N)=\type(L)=\mu(L^{\vee})$. Thus, $\lambda(L^{\vee})=\lambda(L)=e(N)\le 2\type(N)=2\mu(L^{\vee})$. Since $\pd_R(N)<\infty$, one obtains that $\pd_R(L)<\infty$, which implies that $\id_R(L^{\vee})<\infty$, and hence $\pd_R(L^{\vee})<\infty$ as $R$ is in particular Gorenstein. Therefore, by \cite[Thm.~3.8]{IMW}, $e(R)\mu(L^{\vee})\le \lambda(L^{\vee}) \le 2\mu(L^{\vee})$. Thus, since $L^{\vee}\neq 0$, one concludes that $e(R) \le 2$.

(8) $\Longrightarrow$ (1): By Theorem~\ref{thm:pd-e-type-test}, $\pd_R(N)<\infty$. The rest of the proof is similar to that of (6) $\Longrightarrow$ (1).
\end{proof}

In the statements of Theorem~\ref{thm:char-hypersurface-e-le-2}, replacing `$e(N)\le 2\mu(N)$' and `$e(N) \le 2 \type(N) $' by `$e(N) < 2\mu(N)$' and `$e(N) < 2 \type(N) $' respectively, one obtains similar characterizations of regular local rings.

\begin{theorem}\label{thm:char-reg}
The following statements are equivalent.
\begin{enumerate}[\rm (1)]
    \item 
    $R$ is regular.
    \item 
    For every integer $0\le n \le \dim(R)$, there exists a CM $R$-module $N$ of dimension $n$ and minimal multiplicity such that $e(N)< 2\mu(N)$ and $\pd_R(N)<\infty$.
    \item 
    For every integer $0\le n \le \dim(R)$, there exists a CM $R$-module $N$ of dimension $n$ and minimal multiplicity such that $e(N)< 2\type(N)$ and $\id_R(N)<\infty$.
    \item 
   $R$ admits a nonzero CM module $N$ such that $e(N) < 2\mu(N)$ and $\pd_R(N)<\infty$.
    \item 
    $R$ admits a non-zero CM module $N$ such that $e(N)< 2\type(N)$ and $\id_R(N)<\infty$.
    \item
    $R$ admits nonzero CM modules $M$ and $N$ such that $e(M)< 2\mu(M)$, $e(N)< 2\mu(N)$, and there exists $j\ge 0$ satisfying $\Tor^R_n(M,N)=0$ for all $n=j-\dim(M)+1,\ldots,j+\beta_j^{R}(N)$.
    \item 
    $R$ admits nonzero CM modules $M$ and $N$ such that $ e(M) < 2 \mu(M)$, $e(N)< 2 \type(N) $, and there exists $j\ge 0$ satisfying $\Ext_R^n(M,N)=0$ for all $n= j-\dim(M)+1, \ldots, j+s+\sum_{i=0}^s {s \choose i} \mu^{j+i}_R(N)$, where $s=\dim(N)$.
    \item 
    $R$ admits nonzero CM modules $M$ and $N$ such that $e(M) < 2 \type(M) $, $ e(N)< 2 \mu(N)$, and there exists $j \ge 0$ such that $\Ext^n_R(N,M)=0$ for all $n= j+1,\ldots, j+\beta_j^{R}(N)+\dim(M)$.
   
\end{enumerate}   
\end{theorem}

\begin{proof}
    The proof goes exactly through similar lines as that of Theorem~\ref{thm:char-hypersurface-e-le-2} once we notice the following. When $R$ is regular, $e(R)=1$. In the proof of `(6) $\Longrightarrow$ (1)' and `(7) $\Longrightarrow$ (1)', in place of \Cref{DGS.5.5}.(1) and (3), we must use \Cref{DGS.5.5}.(2) and (4) respectively to conclude that $\cx_R(k) = 0$, i.e., $\pd_R(k)<\infty$, and hence  $R$ is regular.
\end{proof}

When $M$ or $N$ or both have minimal multiplicity, one obtains more criteria for a local ring to be regular. We list some of these criteria in the following theorem and remark.

\begin{theorem}\label{thm:char-reg-dim-2}
The following statements are equivalent:
\begin{enumerate}[\rm(1)]
    \item
    $R$ is regular of dimension $\ge 2$.
   \item 
    For every $n$ with $0 \le n \le {\max\{\dim(R)-2,0\}}$, there exists a nonzero CM $R$-module $X$ of dimension $n$ and minimal multiplicity such that $e(X)> 2\type(X)$ and $\pd_R(X)<\infty$.
    \item 
    For every $n$ with $0 \le n \le {\max\{\dim(R)-2,0\}} $, there exists a nonzero CM $R$-module $X$ of dimension $n$ and minimal multiplicity such that $e(X)> 2\mu(X)$ and $\id_R(X)<\infty$.
    \item 
    $R$ admits a non-zero CM module $X$ of minimal multiplicity such that $e(X) > 2\type(X)$ and $\pd_R(X)<\infty$.
    \item 
    $R$ admits a non-zero CM module $X$ of minimal multiplicity such that $e(X) > 2\mu(X)$ and $\id_R(X)<\infty$.
    \item 
    $R$ admits nonzero CM modules $M$ and $N$ such that $e(M)< 2\mu(M)$, $e(N)> 2\type(N)$, $N$ has minimal multiplicity, and there exists $j\ge 0$ satisfying $\Tor^R_n(M,N)=0$ for all $n=j-\dim(M)+1,\ldots,j+\beta_j^R(N)$. 
    \item
    $R$ admits nonzero CM modules $M$ and $N$ such that $e(M)<2\type(M)$, $e(N)> 2\type(N)$, $N$ has minimal multiplicity, and there exists $j \ge 0$ satisfying $\Ext^n_R(N,M)=0$ for all $n=j+1,\ldots,j+\beta_j^R(N)+\dim(M)$.
    \item
    $R$ admits nonzero CM modules $M$ and $N$ such that $e(M)<2\mu(M)$, $e(N)> 2\mu(N)$, $N$ has minimal multiplicity, and there exists $j\ge 0$ satisfying $\Ext_R^n(M,N)=0$ for all $n=j-\dim(M)+1,\ldots,j+s+\sum_{i=0}^s{s\choose i} \mu^{j+i}_R(N)$, where $s:=\dim(N)$.   
\end{enumerate}    
\end{theorem}

\begin{proof}
(1) $\Longrightarrow$ (2) and (1) $\Longrightarrow$ (3): Let $R$ be regular of dimension $d\ge 2$. Then there exists an $R$-regular sequence $x_1,\dots,x_d$ which generates $\fm$. Fix $0\le n\le d-2$. Set ${\bf x}:=x_1,\dots,x_{d-n}$ and $N:=R/({\bf x})^2$. Then, $N$ is a CM $R$-module of projective dimension $d-n$ by \cite[1.4.27 and 2.1.5.(a)]{BH98}. Hence, by the Auslander-Buchsbaum formula, $\dim(N)=d-(d-n)=n$. Set ${\bf y}:=x_{d-n+1},\ldots,x_d$. Since $\dim(N/{\bf y}N)=0$, in view of \cite[2.1.2.(c)]{BH98}, ${\bf y}$ is regular on $N$. Moreover, $\fm^2 N = ({\bf x,y})^2/({\bf x})^2 = ({\bf y})\fm N$. Therefore, by \Cref{lem:min-mult-mod}, $N$ has minimal multiplicity, and $e(N) = \lambda(N/{\bf y}N)$. Note that $\lambda(N/{\bf y}N) = d-n+1\ge 3>2\mu(N)$. Thus $e(N) > 2\mu(N)$.  Since $R$ is regular, one also has that $\id_R(N)<\infty$. Thus the module $X:=N$ satisfies all the conditions in (3). Setting $N^{\dagger}:=\Ext_R^{d-n}(N,R)$, in view of Remark~\ref{rmk:e-type-dagger}, one has that $N^{\dagger}$ is CM $R$-module of dimension $n$, $e(N^{\dagger}) = e(N)$ and $\type(N^{\dagger}) = \mu(N)$. So $e(N^{\dagger}) > 2\type(N^{\dagger})$. As the sequence ${\bf y}$ is regular on $N$, it is also regular on $N^{\dagger}$, see, e.g., \cite[Prop.~5.3.(3)]{GP24}. Since $\fm^2 N = ({\bf y})\fm N$, one has that $\fm^2/({\bf x})^2= ({\bf y})(\fm/({\bf x})^2)$, and then multiplying both sides by $N^{\dagger}$, this yields that $\fm^2N^{\dagger} = ({\bf y})\fm N^{\dagger}$. Thus, by \Cref{lem:min-mult-mod}, $N^{\dagger}$ has minimal multiplicity as an $R$-module. As $R$ is regular, $\pd_R(N^{\dagger})<\infty$. So the module $X:=N^{\dagger}$ satisfies all the conditions in (2).

(2) $\Longrightarrow$ (4) and (3) $\Longrightarrow$ (5): The implications are trivial.

(4) $\Longrightarrow$ (6), (4) $\Longrightarrow$ (7) and (5) $\Longrightarrow$ (8): These implications are obvious by taking $M:=k$ and $N:=X$.

In order to prove the other implications, we may assume that $R$ has infinite residue field.

(6) $\Longrightarrow$ (1): Since $e(M)< 2\mu(M)$, in view of Theorem~\ref{thm:e-mu-pd-id}.(1), $\pd_R(N)<\infty$, i.e., $\cx_R(N) = 0$. Then \Cref{DGS.6.6}.(4) yields that $\cx_R(k)=\cx_R(N)=0$, i.e., $\pd_R(k)<\infty$, and hence $R$ is regular. Since $N$ has minimal multiplicity, by \Cref{lem:min-mult-mod}, there exists a maximal $N$-regular sequence $\mathbf x$ such that $\fm^2N=(\mathbf x)\fm N$ and $e(N)=\lambda(N/(\mathbf x) N)=\lambda(L)$, where $L:=(N/(\mathbf x)N)^{\vee}$. Notice that $\fm^2(N/(\mathbf x)N)=0$, which implies that $\fm^2L=0$. Thus, $L$ is an $R/\fm^2$-module, and hence $\lambda(L)\le \lambda(R/\fm^2)\mu(L)$. By Matlis duality, $\mu(L)=\type(N/(\mathbf x)N)$. If $\dim(R)\le 1$, then $R$ being regular, one has that $\mu(\fm)\le 1$, and hence $\lambda(R/\fm^2)=1+\mu(\fm)\le 2$. Thus, $e(N)=\lambda(L)\le \lambda(R/\fm^2)\mu(L)\le 2\, \type(N/(\mathbf x)N)=2\type(N) $, contradicting $e(N)>2\type(N)$. Therefore $\dim(R)\ge 2$.

(7) $\Longrightarrow$ (1): By Theorem~\ref{thm:pd-e-type-test}, $\pd_R(N)<\infty$. The rest of the proof is similar to that of (6) $\Longrightarrow$ (1).

(8) $\Longrightarrow$ (1): In view of Theorem~\ref{thm:e-mu-pd-id}.(2), $\id_R(N)<\infty$, i.e., $\injcx_R(N) = 0$. Hence \Cref{DGS.6.6}.(2) yields that $\cx_R(k)=\injcx_R(N)=0$. So $R$ is regular. Since $N$ has minimal multiplicity, there exists a maximal $N$-regular sequence $\mathbf x$ such that $\fm^2N=(\mathbf x)\fm N$ and $e(N)=\lambda(N/(\mathbf x) N)=\lambda(L)$, where $L:=N/(\mathbf x)N$. Note that $\fm^2L=0$. If $\dim(R)\le 1$, then $R$ being regular, one has that $\mu(\fm)\le 1$, and hence $\lambda(R/\fm^2)=1+\mu(\fm)\le 2$. Thus, $e(N)=\lambda(L)\le \lambda(R/\fm^2)\mu(L)\le 2 \mu(N)$, contradicting $e(N)>2\mu(N)$. Thus, $\dim(R)\ge 2$.
\end{proof}

\begin{remark}
    In the list of equivalent statements of \Cref{thm:char-reg-dim-2}, we may add the following.
    \begin{enumerate}
        \item[(6$'$)] 
        Replacing only the vanishing condition in \Cref{thm:char-reg-dim-2}.(6) by
        \begin{center}
            $\Tor^R_n(M,N)=0$ for all $n=j-\dim(N)+1,\ldots,j+\beta_j^R(M)$.
        \end{center}
        \item[(8$'$)] 
        Replacing only the vanishing condition in \Cref{thm:char-reg-dim-2}.(8) by
        \begin{center}
            $\Ext_R^n(M,N)=0$ for all $n= j+1, \ldots,j+\beta_j^R(M)+\dim(N)$.
        \end{center}
        \item[(8$''$)] 
        Replacing the vanishing condition in \Cref{thm:char-reg-dim-2}.(8) by
        \begin{center}
            $\Ext_R^n(M,N)=0$ for all $n= j-\dim(M)+1, \ldots, j+\mu_R^j(N)$, for some $j\ge \depth(N)$,
        \end{center}
        and assuming in addition that $M$ also has minimal multiplicity.
    \end{enumerate}
\end{remark}

\begin{proof}
    In order to prove (6$'$) $\Longrightarrow$ (1) and (8$'$) $\Longrightarrow$ (1), by Theorem~\ref{thm:Ext-pd-id-test}.(3) and (1) respectively, one has that $\pd_R(M)<\infty$. Then \Cref{DGS.5.5}.(2) yields that $\cx_R(k)=\cx_R(M)=0$. So $R$ is regular. The rest of the proofs of these implications are similar to that of (6) $\Longrightarrow$ (1) and (8) $\Longrightarrow$ (1) respectively.
    
    (8$''$) $\Longrightarrow$ (1): By \Cref{thm:Ext-pd-id-test}.(2), $\id_R(N)<\infty$. Then \Cref{DGS.6.6}.(2) yields that $\cx_R(k)=\injcx_R(N)=0$. Hence $R$ is regular. Similarly, as in the proof of (8) $\Longrightarrow$ (1), $\dim(R)\ge 2$.
\end{proof}

In addition to \Cref{thm:char-hypersurface-e-le-2}, we obtain the following characterizations of a hypersurface of multiplicity at most $2$ that is not a field.


\begin{theorem}\label{thm:char-hyp-e-le-2-not-field}
The following statements are equivalent:
\begin{enumerate}[\rm(1)]
    \item
    $R$ is a hypersurface of multiplicity at most $2$, but $R$ is not a field.
    \item 
    For every integer $0\le n \le \max\{\dim(R)-1,0\}$, there exists a CM $R$-module $N$ of dimension $n$ and minimal multiplicity such that $e(N)\ge 2\type(N)$ and $\pd_R(N)<\infty$.
    \item 
    For every integer $0\le n \le \max\{\dim(R)-1,0\}$, there exists a CM $R$-module $N$ of dimension $n$ and minimal multiplicity such that $e(N)\ge 2\mu(N)$ and $\id_R(N)<\infty$.
    \item 
    $R$ admits a nonzero CM module $N$ of minimal multiplicity such that $e(N)\ge 2\type(N)$ and $\pd_R(N)<\infty$.
    \item 
    $R$ admits a nonzero CM module $N$ of minimal multiplicity such that $e(N)\ge 2\mu(N)$ and $\id_R(N)<\infty$.
    \item 
    $R$ admits nonzero CM modules $M$ and $N$ such that $e(M)<2\mu(M)$, $e(N)\ge 2\type(N)$, $N$ has minimal multiplicity, and there exists $j\ge 0$ satisfying $\Tor^R_n(M,N)=0$ for all $n=j-\dim(M)+1,\ldots,j+\beta_j^{R}(N)$.
    \item
    $R$ admits nonzero CM modules $M$ and $N$ such that $e(M)<2\type(M)$, $e(N)\ge 2\type(N)$, $N$ has minimal multiplicity, and there exists $j \ge 0$ such that $\Ext^n_R(N,M)=0$ for all $n= j+1,\ldots, j+\beta_j^{R}(N)+\dim(M)$.
    \item
    $R$ admits nonzero CM modules $M$ and $N$ such that $e(M)<2\mu(M)$, $e(N)\ge 2\mu(N)$, $N$ has minimal multiplicity, and there exists $j\ge 0$ satisfying $\Ext_R^n(M,N)=0$ for all $n= j-\dim(M)+1, \ldots, j+s+\sum_{i=0}^s {s \choose i} \mu^{j+i}_R(N)$, where $s=\dim(N)$.    
\end{enumerate}
Moreover, for each of the conditions \rm{(6)--(8)} above, it follows that $\pd_R(N)<\infty$.
\end{theorem}

\begin{proof}
Passing to the faithfully flat extension $R[X]_{\fm[X]}$, we may assume that $R$ has infinite residue field.

(1) $\Longrightarrow$ (2) and (1) $\Longrightarrow$ (3): Set $d:=\dim(R)$. First, assume that $R$ is not regular. Hence $e(R)=2$. Choose an $R$-regular sequence $\mathbf x := x_1,\ldots,x_d$  as in Remark~\ref{rmk:superficial}. Fix $n$ with $0\le n\le d$, and put $N:=R/(x_1,\dots,x_{d-n})$. Then, $\pd_R(N)<\infty$ (hence $\id_R(N)<\infty$ as $R$ is in particular Gorenstein), and $N$ is a CM $R$-module of dimension $n$. Moreover, $2\type(N)=2\type(R)=2=e(R)=e(N)$ and $2\mu(N)=2=e(N)$. That $N$ has minimal multiplicity as an $R$-module follows from Remark~\ref{rmk:superficial}. 

Next, assume that $R$ is regular. Then, since $R$ is not a field, it follows that $R$ is not Artinian, i.e., $d\ge 1$. So one can choose $x\in \fm^2\smallsetminus \fm^3$. Set $R':=R/(x)$. Note that $R'$ is not regular. Pick an $R'$-regular sequence $x_1,\ldots,x_{d-1}$ as in Remark~\ref{rmk:superficial}. Fix $n$ with $0 \le n \le d-1$. Put $N:=R'/(x_1,\dots,x_{d-1-n})$. Then, $\pd_R(N)<\infty$, $\id_R(N)<\infty$, and $N$ is a CM $R$-module of dimension $n$. Also, $e(N)=e(R')=2$, where the last equality follows from \cite[11.2.8]{SH06}. Hence, $e(N)=2=2\type(N)$ and $e(N)=2\mu(N)$. Since $R'$ has minimal multiplicity, so does $N$ by Remark~\ref{rmk:superficial}. 

(2) $\Longrightarrow$ (4)  and (3) $\Longrightarrow$ (5): The implications are trivial.

(4) $\Longrightarrow$ (6), (4) $\Longrightarrow$ (7) and (5) $\Longrightarrow$ (8): These implications are obvious by setting $M:=k$.

(6) $\Longrightarrow$ (1): In view of Theorem~\ref{thm:e-mu-pd-id}.(1), one obtains that $\pd_R(N)<\infty$, i.e., $\cx_R(N)= 0$. Hence \Cref{DGS.6.6}.(3) yields that $\cx_R(k) \le 1$, which implies that $R$ is a hypersurface by Proposition~\ref{prop:cx-curv-facts}.(2). If $R$ is a field, then $N$ is a vector space, hence $e(N)=\type(N)$, a contradiction to the hypothesis. Thus $R$ is not a field. Now, if $e(N)>2\type(N)$, then $R$ is regular by \Cref{thm:char-reg-dim-2}, hence $e(R)=1$. If $e(N)=2\type(N)$, in view of the proof of \Cref{thm:char-hypersurface-e-le-2}.((7) $\Longrightarrow$ (1)), one gets that $e(R)\le 2$.

(7) $\Longrightarrow$ (1): By virtue of Theorem~\ref{thm:pd-e-type-test}, one obtains that $\pd_R(N)<\infty$, i.e., $\cx_R(N)= 0$. Hence the proof is similar as that of (6) $\Longrightarrow$ (1).

(8) $\Longrightarrow$ (1): By Theorem~\ref{thm:e-mu-pd-id}.(2), one gets that $\id_R(N)<\infty$, i.e., $\injcx_R(N) = 0$. Hence, in view of \Cref{DGS.6.6}.(1), $\cx_R(k) \le 1$. So $R$ is a hypersurface. In particular, since $\id_R(N)<\infty$, it follows that $\pd_R(N)<\infty$. If $R$ is a field, then $N$ is a vector space, hence $e(N)=\mu(N)$, a contradiction. Thus $R$ is not a field. Now, if $e(N)>2\mu(N)$, then $R$ is regular by \Cref{thm:char-reg-dim-2}, hence $e(R)=1$. If $e(N)=2\mu(N)$, in view of the proof of \Cref{thm:char-hypersurface-e-le-2}.((6) $\Longrightarrow$ (1)), $e(R)\le 2$. 
\end{proof} 


We conclude this section by providing some criteria for a Gorenstein local ring to be regular in terms of self Ext vanishing involving certain modules of minimal multiplicity.

\begin{corollary}\label{cor:char-RLR-over-Gor}
	The following statements are equivalent for a Gorenstein ring $R$:
    \begin{enumerate}[\rm (1)]
        \item 
        $R$ is regular.
        \item 
        For every $n$ with $0\le n \le \dim(R)$, there exists a nonzero $R$-module $M$ of minimal multiplicity and dimension $n$ such that $e(M)\neq 2\mu(M)$ and $\Ext^{\gg 0}_R(M,M)=0$.
        \item 
        For every $n$ with $0\le n \le \dim(R)$, there exists a nonzero $R$-module $M$ of minimal multiplicity and dimension $n$ such that $e(M)\neq 2\type(M)$ and $\Ext^{\gg 0}_R(M,M)=0$.
        \item 
        $R$ admits a nonzero CM module $M$ of minimal multiplicity such that $e(M) \neq 2\mu(M)$ and $\Ext^{\gg 0}_R(M,M)=0$.
        \item 
        $R$ admits a nonzero CM module $M$ of minimal multiplicity such that $e(M) \neq 2\type(M)$ and $\Ext^{\gg 0}_R(M,M)=0$.
        \item 
        $R$ admits a nonzero module $M$ of minimal multiplicity such that $e(M)\neq 2\mu(M)$ and $\Ext^m_R(M,M)=0$ for  $m=j+1, \ldots, \max\{j+2s+\sum_{i=0}^s {s \choose i} \mu^{j+s+i}_R(M), j+\beta_j^{R}(M)+s\}$ for some $j\ge 0$, where  $s=\dim(M)$.        
        \item 
        $R$ admits a nonzero module $M$ of minimal multiplicity  such that $e(M)\neq 2\type(M)$ and $\Ext^m_R(M,M)=0$ for $m$ varies in the same range as in {\rm (6)}.
    \end{enumerate}
\end{corollary}

\begin{proof}
The implications (2) $\Longrightarrow$ (4) $\Longrightarrow$ (6) and (3) $\Longrightarrow$ (5) $\Longrightarrow$ (7) are obvious.

(6) $\Longrightarrow$ (1): First, assume that $e(M)< 2\mu(M)$. Then, since
$\Ext^m_R(M,M)=0$ for $m=j+1, \ldots, j+2s+\sum_{i=0}^s {s \choose i} \mu^{j+s+i}_R(M)$, in view of \Cref{thm:e-mu-pd-id}.(2), $\id_R(M)<\infty$. Thus, since $R$ is Gorenstein, $\pd_R(M)<\infty$. Therefore, by \cite[Cor.~7.15]{DGS}, $R$ is regular. In the other case, assume that $e(M)>2\mu(M)$. Then, since $\Ext^m_R(M,M)=0$ for $m=j+1, \ldots, j+\beta_j^{R}(M)+s$, \Cref{thm:Ext-pd-id-test}.(1) yields that $\pd_R(M)<\infty$.  As $R$ is Gorenstein, it follows that $\id_R(M)<\infty$. Therefore, by \Cref{DGS.6.6}.(2), $\cx_R(k)=0$, i.e., $\pd_R(k)<\infty$, and hence $R$ is regular.  

(7) $\Longrightarrow$ (1): First, assume that $e(M) > 2\type(M)$. Then, since $\Ext^n_R(M,M)=0$ for $n=j+1, \ldots, j+2s+\sum_{i=0}^s {s \choose i} \mu^{j+i}_R(M)$, by \Cref{thm:Ext-pd-id-test}.(4) yields that $\id_R(M)<\infty$. Thus, since $R$ is Gorenstein, $\pd_R(M)<\infty$. Hence, by \Cref{DGS.6.6}.(4), $R$ is regular. Next, assume that $e(M)< 2\type(M)$. Then, since $\Ext^m_R(M,M)=0$ for $m=j+1, \ldots, j+\beta_j^{R}(M)+s$, in view of \Cref{thm:pd-e-type-test}, $\pd_R(M)<\infty$.  So $\id_R(M)<\infty$ as $R$ is Gorenstein. Finally, by \Cref{DGS.5.5}.(4), $R$ is regular.

(1) $\Longrightarrow$ (2) and (1) $\Longrightarrow$ (3): Set $d:=\dim(R)$. The maximal ideal $\fm$ is generated by a regular sequence $x_1,\ldots,x_d$. For $0\le n\le d$, set $M:=R/(x_1,\ldots,x_{d-n})$. Then, $M$ is CM of dimension $n$, and $e(M)=1$ as $R/(x_1,\ldots,x_{d-n})$ is a regular ring. In particular, $M$ has minimal multiplicity, and it satisfies $e(M)=1\neq 2= 2\mu(M)= 2\type(M)$. The Ext vanishing is obvious as $R$ is regular. 
\end{proof}

\begin{remark}
    The assumption that $R$ is Gorenstein in Corollary~\ref{cor:char-RLR-over-Gor} cannot be removed. Consider an Artinian local ring $(R,\fm)$ such that $\fm^2=0$ and $\mu(\fm)\ge 2$, see \cite[Example~5.8]{DGS}. Then $M:=R$ satisfies all the conditions (2) - (7) of \Cref{cor:char-RLR-over-Gor}, however $R$ is not regular.
\end{remark}

\begin{remark}
    The conditions `$ e(M) \neq 2 \mu(M) $' and `$ e(M) \neq 2 \type(M) $' cannot be omitted from \Cref{cor:char-RLR-over-Gor}. As in Example~\ref{exam:strict-ineq-cant-be-eq}, let $R$ be a hypersurface of multiplicity $2$. Set $M:=R$. Then $M$ is a CM $R$-module of minimal multiplicity such that $e(M)=2=2\mu(M)=2\type(M)$. Clearly, $\Ext_R^{\ge 1}(M,M)=0$. However, $R$ is not regular.
\end{remark}

\begin{remark}
    Without the condition `minimal multiplicity', \Cref{cor:char-RLR-over-Gor} does not necessarily hold true. Let $R$ be a Gorenstein Artinian local ring of Loewy length $\ge 3$ (e.g., $R=k[x]/(x^n)$ for some $n\ge 3$). Then $M:=R$ satisfies all the conditions (2) - (7) in \Cref{cor:char-RLR-over-Gor} except that $M$ does not have minimal multiplicity. Here $R$ is not regular.
\end{remark}

\section{Characterizations of Gorenstein and CM local rings}\label{sec:char-Gor-CM}

We show a few characterizations of Gorenstein and CM local rings as well. In the following result, comparing the conditions (2) and (3), note that ``$\Ext_R^{\gg 0}(M,R) = 0$" is a much weaker condition than ``$\gdim_R(M)< \infty$", see \cite[p.~218]{JS06}.

\begin{corollary}\label{cor:char-Gor}
    The following statements are equivalent.
    \begin{enumerate}[\rm (1)]
        \item $R$ is Gorenstein.
        \item $R$ admits a nonzero CM module $M$ of minimal multiplicity such that $e(M) \le 2 \mu(M)$ and $\gdim_R(M)< \infty$.
        \item 
        $R$ admits a CM module $M$ such that $e(M) < 2 \mu(M)$ and $\Ext_R^{\gg 0}(M,R) = 0$.
        \item 
        $R$ admits two nonzero modules $M$ and $N$ such that $e(M) < 2 \mu(M)$, $\pd_R(N)<\infty$ and $\Ext_R^{\gg 0}(M,N) = 0$. Suppose that $M$ is CM, and that $R$ or $N$ is CM.
    \end{enumerate}
\end{corollary}

\begin{proof}
	If $R$ is Gorenstein, then $M=k$ satisfies all conditions in (2), (3) and (4).
    
    (2) $\Longrightarrow$ (1): We may assume that $k$ is infinite. In view of \Cref{lem:min-mult-mod}.(3), there exists a system of parameters ${\bf x}$ of $M$ such that $\fm^2M = ({\bf x}) \fm M$ and $e(M) = \lambda(M/{\bf x}M)$. Since $M$ is CM, ${\bf x}$ is also an $M$-regular sequence. Thus, setting $N:=M/{\bf x}M$, we have that $\fm^2N=0$, $\lambda(N) = e(M) \le 2 \mu(M) = 2 \mu(N)$ and $\gdim_R(N)< \infty$. Therefore, by \cite[Cor.~2.6]{Ta04}, $R$ is Gorenstein.

    (3) $\Longrightarrow$ (1): By \Cref{thm:e-mu-pd-id}.(1), $M$ is a Tor-pd-test $R$-module. So, by \cite[Thm.~4.4]{CS16}, $R$ is Gorenstein.

    (4) $\Longrightarrow$ (1): In view of Theorem~\ref{thm:e-mu-pd-id}.(2) and (3), $\id_R(N)<\infty$. Hence, by \cite[Cor.~4.4]{Fo77}, $R$ is Gorenstein.
\end{proof}





A natural question arises whether the condition `$e(M) < 2 \mu(M)$' in Corollary~\ref{cor:char-Gor}.(3) can be replaced by `$e(M) \le 2 \mu(M)$'.

\begin{question}\label{ques:char-Gor-via-Ext-vanishing}
    If $R$ admits a nonzero CM module $M$ of minimal multiplicity such that $e(M) = 2 \mu(M)$ and $\Ext_R^{\gg 0}(M,R) = 0$, then is $R$ Gorenstein, or equivalently, is $\gdim_R(M)$  finite?
    
    Equivalently, if $R$ admits a nonzero module $M$ such that $\fm^2 M=0$, $ \lambda(M) = 2 \mu(M) $ and $\Ext_R^{\gg 0}(M,R) = 0$, then is $R$ Gorenstein, or equivalently, is $\gdim_R(M)$  finite?
\end{question}

We observe that \Cref{ques:char-Gor-via-Ext-vanishing} has a positive answer for Golod rings. For that, we first record a lemma.

\begin{lemma}\label{lem:Golod-Ext-free}
Let $R$ be a Golod ring, which is not a hypersurface. Let $M$ be an $R$-module. If $\Ext_R^i(M,R)=0$ for all $1\le i\le 2\mu(\fm)-\depth(M)$, then $M$ is free. 
\end{lemma}

\begin{proof}
Set $b:=\mu(\fm)$. As $R$ is not a hypersurface, it follows that $b-\depth(R)\ge 2$ (see, e.g., \cite[5.1]{Avr98}). Moreover, $\depth(M)\le \dim(R)\le b-1$, where the last inequality holds as $R$ is not regular. If possible, let $\pd_R(M)=\infty$. Set $N:=\Omega_R^{b-\depth(M)-1}(M)$. Let $\{F_{i}\}_{i\ge 0}$ be a minimal free resolution of $N$. Then $\mu(F_i) = \beta_{b-\depth(M)-1+i}^R(M)$ for all $i\ge 0$. So, using \cite[5.3.3.(5)]{Avr98} for $M$, one obtains that $\mu(F_{i+1}) > \mu(F_i)$ for all $i\ge 1$. From the hypothesis, note that $\Ext_R^{1\le i\le b+1}(N,R)=0$. Set $n:=b+2$. Dualizing the exact sequence $F_n\xrightarrow{\Phi_n} F_{n-1}\to \dots\to F_2 \xrightarrow{\Phi_2} F_1\xrightarrow{\Phi_1} F_0\to N\to 0$ by $R$, one gets another exact sequence
$$0\to N^*\to F_0\xrightarrow{\Phi_1^t} F_1\xrightarrow{\Phi_2^t} F_2 \to \dots \to F_{n-1} \xrightarrow{\Phi_n^t} F_n \to X \to 0$$
for some $R$-module $X$. Since each $\Phi_i$ has entries in $\fm$, so does its transpose $\Phi_i^t$, and hence $\mu(F_1)$ and $\mu(F_2)$ are the $(n-1)$st and $(n-2)$nd Betti numbers of $X$ respectively. Since $\pd_R(N)=\infty$, by \cite[Prop.~3.15.(2)]{DG24}, it follows that $\pd_R(N^*)=\infty$. Therefore $\pd_R(X)=\infty$. Since $n-2\ge b$, \cite[5.3.3.(5)]{Avr98} yields that $\mu(F_1)>\mu(F_2)$. This contradicts that $\mu(F_2)>\mu(F_1)$. Thus, we must have that $\pd_R(M)<\infty$. Clearly, $\pd_R(M)\le b$. Since $2b-\depth(M)\ge b$, now \cite[Lem.~1.(iii), p.~154]{Mat86} implies that $M$ is free.
\end{proof}  




\begin{proposition}
Let $Q$ be a Golod ring, and $x_1,\ldots,x_c$ be a $Q$-regular sequence. Set $R=Q/(x_1,\ldots,x_c)$. Suppose there exists a nonzero CM $R$-module $M$ such that $e_R(M) \le 2 \mu_R(M)$ and $\Ext_R^{\gg 0}(M,R) = 0$. Then, $R$ is a complete intersection of codimension $\le c+1$. 
\end{proposition}

\begin{proof}
We first prove the result when $c=0$, i.e., $R$ itself is Golod. In this case, we show that $R$ is hypersurface. If possible, suppose that $R$ is not hypersurface. Given $\Ext_R^{\gg 0}(M,R) = 0$, i.e., there exists $n \ge 0$ such that $\Ext_R^i(M,R) = 0$ for all $i \ge n$. Then $\Ext_R^i(\Omega_R^n(M),R) = 0$ for all $i\ge 1$.
So, by \Cref{lem:Golod-Ext-free}, $\Omega_R^n(M)$ is free. Hence $\pd_R(M) < \infty$. This contradicts the implication (4) $\Longrightarrow$ (1) of \Cref{thm:char-hypersurface-e-le-2}. Therefore, $R$ is hypersurface.
    
Now we consider the general case, i.e., $R=Q/(x_1,\ldots,x_c)$. Since $\Ext_R^{\gg 0}(M,R) = 0$, it follows that $\Ext_Q^{\gg 0}(M,Q) = 0$, see, e.g., \cite[p.~140]{Mat86}. Note that $M$ is a nonzero CM $Q$-module and $e_Q(M)=e_R(M) \le 2 \mu_R(M)=\mu_Q(M)$. By the $c=0$ case, $Q$ is a hypersurface, hence $R$ is a complete intersection of codimension at most $c+1$.    
\end{proof}

We record affirmative answers to \Cref{ques:char-Gor-via-Ext-vanishing} for further special cases in the following remark.

\begin{remark}
    Question~\ref{ques:char-Gor-via-Ext-vanishing} has an affirmative answer when $R$ is CM, and it satisfies at least one of the following conditions:
    \begin{enumerate}[\rm (1)]
        \item $R$ has codimension $\le 3$.
        \item $R$ is one link from a complete intersection, and it admits a canonical module.
        \item $R$ is radical cube zero.
    \end{enumerate}
\end{remark}

\begin{proof} Assume the hypotheses as in \Cref{ques:char-Gor-via-Ext-vanishing}.
    If necessary, passing through the completion for case (1), without loss of generality, we may assume, in case (1) also, that $R$  admits a canonical module, say $\omega$. Note that $\omega$ is MCM, and $\Hom_R(\omega,\omega) \cong R$. Therefore, considering $N=\omega$ in \cite[Thm.~5.2]{GP24}, since $\Ext_R^{\gg 0}(M,R) = 0$, one obtains that $\Tor_{\gg 0}^R(M,\omega) = 0$. Hence, by \cite[Thm.~5.3.(1).(ii)]{DGS}, $\cx_R(\omega) \le 1$, i.e., the Betti sequence of $\omega$ is bounded. Finally, in view of \cite[Prop.~1.1]{JL07}, one concludes that $R$ is Gorenstein.
\end{proof}

Related to Corollary~\ref{cor:char-Gor}, we note the following existing results in the literature.

\begin{para}\label{para:char-Gor}
	In \cite[Thm.~3.1]{Ulr84}, Ulrich proved that a CM local ring $R$ is Gorenstein if there is an $ R $-module $ L $ of positive rank such that $ e(L) < 2 \nu(L) $ and $ \Ext_R^{1 \le i \le d}(L,R) = 0 $. Hanes-Huneke in \cite[Thms.~2.5 and 3.4]{HH05} and Jorgensen-Leuschke in \cite[Thms.~2.2 and 2.4]{JL07} gave some analogous criteria.
	Recently, Lyle and Monta\~{n}o \cite[Thm.~D]{LM20} showed that a CM generically Gorenstein local ring $R$
    is Gorenstein if there is an MCM $R$-module $L$ such that $e(L) \le 2 \nu(L)$ and $ \Ext_R^{1 \le i \le d+1}(L,R) = 0 $.
\end{para}

The following corollary provides a characterization of CM local rings in terms of vanishing of Ext.

\begin{corollary}\label{cor:char-CM}
    The following statements are equivalent.
    \begin{enumerate}[\rm (1)]
        \item $R$ is CM, and it admits a dualizing module.
        \item $R$ admits a CM module $M$ of minimal multiplicity and satisfying $e(M) < 2 \mu(M)$, and there exists a semidualizing $R$-module $C$ such that $\Ext_R^{\gg 0}(M,C)=0$.
    \end{enumerate}
\end{corollary}

\begin{proof}
  (1) $\Longrightarrow$ (2): If $\omega$ is a canonical module of $R$, then $M:=k$ satisfies (2) with $C:=\omega$.
  
  (2) $\Longrightarrow$ (1):
  In view of \Cref{thm:Ext-pd-id-test}.(2), $\id_R(C) < \infty$, and hence $C$ is a dualizing module. Also, by Bass' conjecture, $R$ is CM.
\end{proof}




\section{Application to Auslander-Reiten conjecture}\label{sec:ARC-CM-MIN}

In this section, we mainly prove that \Cref{ARC} holds true for CM modules of minimal multiplicity over an arbitrary Noetherian local ring.

\begin{theorem}\label{thm:ARC-finite-pd}
    Let $M$ be a nonzero CM $R$-module. Then, $R$ is CM, and $\pd_R(M)<\infty$ if one of the following conditions holds.
    \begin{enumerate}[\rm (1)]
        \item 
        $ e(M) < 2 \mu(M)$, and there exists $j,l \ge 0$ such that
        \begin{enumerate}[]
            \item $\Ext_R^m(M,R)=0$ for all $m = j-\dim(M)+1,\ldots,j+d+\sum_{i=0}^d {d\choose i} \mu^{j+i}_R(R)$, where $d:=\dim(R)$, and
            \item $\Ext_R^n(M,M)=0$ for all $n = l-\dim(M)+1,\ldots,l+s+\sum_{i=0}^s {s\choose i} \mu^{l+i}_R(M)$, where $s=\dim(M)$.
        \end{enumerate}
        \item 
        $ e(M) = 2 \mu(M) $, $M$ has minimal multiplicity, and $\Ext_R^n(M,M)=0$ for all $n\gg 0$.
        \item 
        $e(M) > 2 \mu(M) $, $M$ has minimal multiplicity, and there exists $l \ge 0$ such that $\Ext_R^n(M,M)=0$ for all $n= l+1, \ldots, l+\beta_l^{R}(M)+\dim(M)$.      
    \end{enumerate}
\end{theorem}

\begin{proof}
    We first note that $\pd_R(M)<\infty$ would imply that $R$ is CM by \cite[9.4.6 and 9.4.8.(a)]{BH98}. So, we only need to prove that $\pd_R(M)<\infty$.

    When $M$ satisfies (1), in view of Theorem~\ref{thm:e-mu-pd-id}.(2), $\id_R(M) < \infty$. Therefore, by Bass' Conjecture, $R$ is CM. Again, by using Theorem~\ref{thm:e-mu-pd-id}.(2), one obtains that $\id_R(R)<\infty$, hence $R$ is Gorenstein. It follows that $\pd_R(M) < \infty$.
    
    For (2), by \cite[Cor.~7.2]{DGS}, $R$ is complete intersection, and hence \cite[Cor.~4.4]{AY98} yields that $\pd_R(M)<\infty$.

     When $M$ satisfies (3), in view of Theorem~\ref{thm:Ext-pd-id-test}.(1) one obtains that $\pd_R(M)<\infty$.
\end{proof}

The theorem below provides various sufficient conditions for a CM module of minimal multiplicity to be free in terms of vanishing of certain Ext modules. The reader may compare Theorems~\ref{thm:ARC-finite-pd}.(1) and \ref{thm:ARC-free-criteria}.(1). None of them follows from the other.

\begin{theorem}\label{thm:ARC-free-criteria}
    Let $M$ be a nonzero CM $R$-module of minimal multiplicity. Furthermore, assume that one of the following conditions holds.
    \begin{enumerate}[\rm (1)]
        \item $ e(M) < 2 \mu(M) $, $\Ext_R^m(M,R)=0$ for all $1\le m \le \max\big\{\depth(R)+\type(R),\, \dim(M)+\mu_R^{\dim(M)}(R)\big\}$, and $\Ext_R^n(M,M)=0$ for all $1\le n \le \dim(M)+\type(M)$.
        \item $ e(M) = 2 \mu(M) $ and $\Ext_R^n(M,M)=0$ for all $n\ge 1$.
        \item $ e(M) > 2 \mu(M) $ and $\Ext_R^n(M,M)=0$ for all $$1\le n \le \max\{\mu(M)+\dim(M),\,\depth(R)\}.$$
    \end{enumerate}
    Then, $M$ is free, and $R$ is CM of minimal multiplicity.
\end{theorem}

\begin{proof}
    In view of \cite[Rmk.~4.4]{DGS} and Definition~\ref{defn:ring-min-mult}, one derives that if a free $R$-module is CM of minimal multiplicity, then the ring $R$ is CM of minimal multiplicity. So we only need to prove that $M$ is free.
    
    If $M$ satisfies (2) or (3), in view of Theorem~\ref{thm:ARC-finite-pd}.(2) and (3) respectively, $\pd_R(M)<\infty$. Hence it follows from \cite[p.~154, Lem.~1.(iii)]{Mat86}
    that $M$ is free.

   Suppose that $M$ satisfies (1). Then, by Theorem~\ref{thm:Ext-pd-id-test}.(2) for $N:=M$ and $j=\dim(M)$, one gets that $\id_R(M)<\infty$. Depending on whether $\depth(R)>\dim(M)$ or $\depth(R)\le \dim(M)$, considering $j=\depth(R)$ or $j=\dim(M)$ respectively, Theorem~\ref{thm:Ext-pd-id-test}.(2) for $N:=R$ yields that $\id_R(R)<\infty$, i.e., $R$ is Gorenstein. Therefore, since $\id_R(M)<\infty$, one obtains that $\pd_R(M)<\infty$. Hence, as $\Ext_R^j(M,R)=0$ for all $1\le j \le \depth(R)$, it follows that $M$ is free.
\end{proof}

In particular, Theorems~\ref{thm:ARC-free-criteria} and \ref{thm:ARC-finite-pd}, respectively, prove \Cref{ARC}.(1) and (2) for every CM module of minimal multiplicity over a Noetherian local ring.

\end{document}